\pdfoutput=1
\documentclass[10pt,leqno]{amsart}

\usepackage[T1]{fontenc}
\usepackage[utf8]{inputenc}

\usepackage{graphicx}
\usepackage{indentfirst,csquotes}

\usepackage{amssymb,amsthm,amsmath}
\usepackage{array}
\usepackage{booktabs}
\usepackage{bussproofs}
\usepackage{tikz}
\usetikzlibrary{arrows.meta,positioning}
\usepackage{xcolor,paralist,hyperref,fancyhdr,etoolbox}
\providecommand{\texorpdfstring}[2]{#1}

\theoremstyle{plain}
\newtheorem{thm}{Theorem}[section]
\newtheorem{prop}[thm]{Proposition}
\newtheorem{lem}[thm]{Lemma}
\newtheorem{cor}[thm]{Corollary}
\theoremstyle{definition}
\newtheorem{dfn}[thm]{Definition}
\newtheorem{exa}[thm]{Example}
\newtheorem{con}[thm]{Convention}
\theoremstyle{remark}
\newtheorem{rem}[thm]{Remark}
\numberwithin{equation}{section}

\hypersetup{ colorlinks=true, linkcolor=black, filecolor=black, urlcolor=black,
             citecolor=black }

\newcommand{\seq}{\Rightarrow}
\newcommand{\arr}{\rightarrow}
\newcommand{\sT}{\mathsf{T}}
\newcommand{\sF}{\mathsf{F}}
\newcommand{\Ptt}{\Pi^{\sT}}
\newcommand{\bti}{\mathsf{BT}^{i}}
\newcommand{\btd}{\mathsf{BT}^{i}_{\downarrow}}

\newcommand{\Gti}{\mathsf{G3T}^{i}}
\newcommand{\Gzi}{\mathsf{G0T}^{i}}
\newcommand{\Ngi}{\mathsf{NgT}^{i}}
\newcommand{\Gip}{\mathsf{G3ip}}
\newcommand{\Gfip}{\mathsf{G4ip}}
\newcommand{\LK}{\mathrm{LK}}
\newcommand{\LJ}{\mathrm{LJ}}
\newcommand{\NJ}{\mathrm{NJ}}

\newcommand{\wkL}{\mathrm{wkL}}
\newcommand{\wkR}{\mathrm{wkR}}
\newcommand{\ctrL}{\mathrm{ctrL}}
\newcommand{\ctrR}{\mathrm{ctrR}}
\newcommand{\Sub}{\mathrm{Sub}}
\newcommand{\Subpm}{\mathrm{Sub}^{\pm}}
\newcommand{\oa}{\mathrm{oa}}
\newcommand{\ef}{\mathrm{end}}
\newcommand{\wt}{\mathrm{w}}
\newcommand{\forces}{\Vdash}
\newcommand{\nforces}{\nVdash}

\newcommand{\acceptednote}{%
  \begin{center}
  \fbox{\begin{minipage}{0.88\textwidth}
  \centering\small
  \textsc{Accepted for publication in} \textit{Reports on Mathematical Logic}\\[3pt]
  \footnotesize This is the author's accepted manuscript.
  The version of record will appear in \textit{Reports on Mathematical Logic} \textbf{61} (2026).
  \end{minipage}}
  \end{center}
  \medskip}

\begin{document}

\title[Sequent-style tableaux for intuitionistic propositional logic]{Sequent-style tableaux for\\ intuitionistic propositional logic}

\author[S.\ Cuconato]{Simone Cuconato}

\address{Department of Physics,
\newline \indent University of Calabria,
\newline \indent 87036 Rende (CS), Italy}
\email{simone.cuconato@unical.it
\newline \indent ORCID iD: \url{https://orcid.org/0000-0003-0277-9575}}

\thanks{\textit{Accepted for publication in} \textbf{Reports on Mathematical Logic}.
This preprint is the author's accepted manuscript; the version of record will
appear in \emph{Reports on Mathematical Logic} \textbf{61} (2026).}



\date{}

\begin{abstract}
Sequent-style tableaux are a refutation calculus in which each node of the
refutation tree carries a finite \emph{block} of formulae and the structural
rules are absorbed into the data structure and the closure criterion.  In their
original, classical form they rest on an involutive De Morgan negation and on
closure upon a complementary pair.  We show that both may be dispensed with.
Replacing unsigned formulae by \emph{signed} ones, we obtain a block calculus
$\bti$ for intuitionistic propositional logic in which the whole of
intuitionism is carried by one rule, the rule decomposing $\sF(A\arr B)$, which
deletes the $\sF$-part of the context on passing to the child block.  The rules
so obtained are, up to the presentation, those of Fitting's signed tableaux;
what is new is the block format, in which the structural rules are absorbed
rather than admissible, and what follows from it.  We
identify the semantic reason for this rule and for the one other anomalous one: of
the signed compounds of the language, exactly those governed by the implication
fail to be locally decomposable, and the two failures are repaired,
respectively, by retaining the principal formula and by purging the context.  We
prove that $\bti$ is the multiple-succedent sequent calculus $\Gti$ read upside
down, that $\Gti$ admits the structural rules, and that $\bti$ is sound and
complete for Kripke semantics, with the finite model property and a
block-theoretic proof of the disjunction property.  We then transcribe the
calculus into a G0-style sequent calculus $\Gzi$ and into a natural deduction
system $\Ngi$ with general elimination rules, for which we prove a full
normalization theorem by bidirectional translation.  A depth-reducing variant
$\btd$, obtained by transcribing Dyckhoff's contraction-free calculus, restores
termination of refutation search.  A concluding section compares the three
formalisms and locates, in each, the exact place at which intuitionism resides.
\end{abstract}

\maketitle

\acceptednote

\section{Introduction}\label{sec:intro}

The calculus of \emph{sequent-style tableaux} was introduced as a hybrid
\cite{CuconatoIM,cuconato2025metodi}: it grafts the notation of Gentzen's
sequent calculus onto the branching architecture of analytic tableaux.  Each
node of a refutation tree carries not a single formula, as in Smullyan's
presentation \cite{Smullyan1968}, but a finite set $\Pi$ of formulae, a \emph{block}. The block associated with a sequent $\Gamma\seq\Delta$ is $\Gamma\cup\lnot[\Delta]$, so that the
sequent arrow is reabsorbed into a one-sided list. A branch does not accumulate
formulae as it descends: it transports the whole block, decomposing at each step
a principal formula and inheriting the rest unchanged. The calculus has no cut rule, and this is not a theorem about it but a feature of its design: no rule introduces, on passing
from a block to its children, a formula that is not a component of the
principal one, so there is nowhere for a cut to sit.  Its structural rules are absorbed rather than eliminated: contraction into the set-theoretic reading of blocks, weakening into the closure criterion.  It stands in a bidirectional correspondence with cut-free $\LK$ \cite{CuconatoIM}, and on the strength of that correspondence it has been inserted into the
architecture of G3- and G0-style sequent calculi and general-elimination natural
deduction \cite{CuconatoKJM}, in the manner of Kamide and Negri \cite{KamideNegri2025}.

All of this is classical, and it is classical for two reasons which it is worth
keeping apart.  The first is the treatment of negation: $\lnot\lnot A$ is
deductively identified with $A$, and the negations of the compounds are
governed by the De Morgan laws.  The second is the closure criterion: a block
closes as soon as it contains a complementary pair $\theta,\lnot\theta$, and
the block itself makes no distinction between the formulae that come from the
antecedent and those that come from the negated succedent.  The question that
occasions the present paper is whether the design survives the removal of both,
and what has to be paid for its survival.

Our answer is that the design survives essentially intact, and that the price
is remarkably local.  Signed formulae $\sT A$ and $\sF A$ replace the
involution: the sign does, at no cost, the work that the De Morgan laws did in
the classical calculus, and it does it while keeping every rule analytic.  The
closure criterion becomes closure on a pair $\sT p,\sF p$, or on $\sT\bot$.
With this much fixed, seven of the nine propositional rules are exactly what one
expects, and the whole of intuitionism is concentrated in the remaining two,
those that decompose an implication under the sign $\sF$: these delete, on
passing to the child block, every $\sF$-signed formula of the context.  We call
this deletion the \emph{purge}.  It is the block-level image of the rule
$\arr\!\mathrm{R}$ of the multiple-succedent intuitionistic sequent calculus of
Maehara \cite{Maehara1954} and Dragalin \cite{Dragalin1988}, the rule that
empties the succedent in its premiss; and the semantic reason for it is the
monotonicity of forcing, which licenses the transport of the $\sT$-signed
formulae to a later world and forbids the transport of the $\sF$-signed ones.

One structural observation orients everything that follows.  In the classical
first-order block calculus the two features that break the uniformity of the
propositional rules are the \emph{persistence} of the principal formula in the
rules of type $\gamma$ and the \emph{freshness} of the parameter in the rules of
type $\delta$.  Both reappear here, at the propositional level and without a
quantifier in sight.  The rule for $\sT(A\arr B)$ retains its principal formula
in one of its children --- which is precisely the repetition that Dragalin and
Troelstra introduce in $\arr\!\mathrm{L}$ in order to make contraction
admissible --- and the rule for $\sF(A\arr B)$ constrains the context it passes
on.  Relative to the classical calculus, then, intuitionistic propositional
logic behaves like a quantificational, or modal, extension of it.  Through the
G\"odel--McKinsey--Tarski translation this is not an analogy but a fact, and we
return to it in Section~\ref{sec:remarks}.

It is right to say at once what is inherited here and what is not. The rules of $\bti$ are, up to the presentation, those of the unprefixed signed tableau system for intuitionistic propositional logic given by Fitting \cite{Fitting1969}. The signs themselves, the branching rules for conjunction and disjunction, the survival of an implication under the sign $\sT$, and above all the deletion of the $\sF$-signed part in the rules governing an implication
under the sign $\sF$: all of these are to be found there.  So, in consequence,
is the adequacy of the system for Kripke semantics, which we prove again below
only because the block presentation and the sequent correspondence require the
proofs in a particular form. That such tableaux and the multiple-succedent calculi of
Maehara \cite{Maehara1954} and Dragalin \cite{Dragalin1988} are two faces of one
system is likewise not a discovery of ours; it is recorded, in one form or
another, in \cite{Fitting1983} and \cite{TroelstraSchwichtenberg2000}.

What the present paper adds is of four kinds.  First, the presentation.  A block is a set and its principal formula is consumed, so that the structural
rules are not merely admissible but absorbed: into the data structure, into the
closure criterion, and into the survival of the implication.  It is exactly
because everything else is consumed that this last survival becomes a marked
exception, visible and open to analysis.  In a system whose rules only add to a
set it passes unremarked. Second, and this we take to be the substance of the paper,
an explanation of the two anomalous rules in place of a stipulation of them:
Proposition~\ref{prop:locdec} isolates the property of \emph{local
decomposability}, shows that among the signed compounds exactly the two governed
by the implication lack it, and shows that the two failures are of different
kinds and call for two different repairs, retention in the one case and the
purge in the other; the same proposition then serves as a criterion for
admitting further connectives to the format. Third, the insertion of the
calculus into the architecture of G0-style sequent calculi and
general-elimination natural deduction, in the manner of Kamide and Negri
\cite{KamideNegri2025}, with the full normalization theorem of
Section~\ref{sec:nd}; no natural deduction counterpart of the intuitionistic
tableau system seems to have been given before.  Fourth, three smaller items: a
proof of the disjunction property internal to the block format
(Theorem~\ref{thm:dp}), the transcription of Dyckhoff's contraction-free
calculus into blocks together with a termination measure verified rule by rule
(Theorem~\ref{thm:term}), and the comparison of Section~\ref{sec:comp}.  The
adequacy of that depth-reducing variant is Dyckhoff's theorem
\cite{Dyckhoff1992}, and we cite it rather than reprove it.

The plan of the paper is as follows.  Section~\ref{sec:lang} fixes the language
and the Kripke semantics.  Section~\ref{sec:block} introduces the block
calculus $\bti$, isolates its two intuitionistic rules, derives the rules for
negation, establishes the subformula property and the admissibility of
weakening, and works three examples.  Section~\ref{sec:g3} introduces the
multiple-succedent sequent calculus $\Gti$, proves that the block calculus is
$\Gti$ read upside down, and establishes the structural properties of $\Gti$ up
to the admissibility of cut.  Section~\ref{sec:sem} proves soundness and
completeness with respect to Kripke models, and draws three corollaries: the
finite model property with decidability, the disjunction property by a purely
block-theoretic argument, and the collapse of multiple to single succedents.
Section~\ref{sec:nd} transcribes the calculus into a G0-style sequent calculus
$\Gzi$, obtains cut elimination for it by transfer, introduces the natural
deduction system $\Ngi$ with general elimination rules, and proves the full
normalization theorem by bidirectional translation.  Section~\ref{sec:term}
gives a depth-reducing variant $\btd$ of the block calculus, transcribing
Dyckhoff's contraction-free calculus \cite{Dyckhoff1992}, and verifies a
termination measure rule by rule.  Section~\ref{sec:comp} compares the three
formalisms.  Section~\ref{sec:remarks} collects concluding remarks.

The model for the architecture of the paper is Ili\'c's study of an alternative
natural deduction for intuitionistic propositional logic \cite{Ilic2016}, whose
sequent calculus $GI$ is the closest published relative of the system $\Gzi$ of
Section~\ref{sec:nd}: the two agree on the treatment of implication and differ
in the handling of conjunction and disjunction, with consequences for the
natural deduction systems they generate.  We take up the comparison in
Section~\ref{sec:comp}.  For the classical tableau tradition on which the block
calculus draws we refer to Smullyan \cite{Smullyan1968}, and for the history of
intuitionistic logic itself to Troelstra \cite{Troelstra1988}.

Figure~\ref{fig:network} records the calculi and the results that tie them
together, and may be used as a map of what follows.

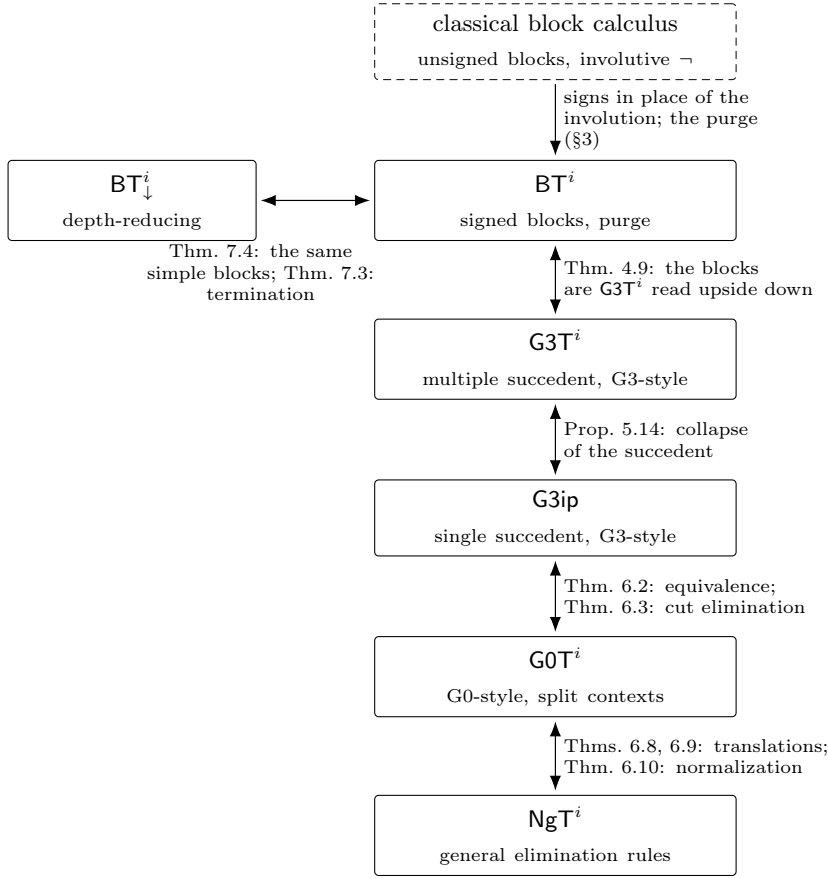
\begin{figure}[htb]
\centering
\begin{tikzpicture}[
  >={Latex[length=2.0mm]},
  bx/.style={draw, rounded corners=1.5pt, align=center, inner sep=4pt,
             minimum height=9mm, text width=45mm, line width=0.4pt},
  sm/.style={draw, rounded corners=1.5pt, align=center, inner sep=4pt,
             minimum height=9mm, text width=30mm, line width=0.4pt},
  gh/.style={draw, densely dashed, rounded corners=1.5pt, align=center,
             inner sep=4pt, minimum height=9mm, text width=45mm,
             line width=0.35pt},
  lab/.style={font=\scriptsize, align=left, inner sep=3pt, text width=36mm},
  slab/.style={font=\scriptsize, align=center, inner sep=2pt, text width=34mm},
  eq/.style={<->, line width=0.45pt, shorten >=1.5pt, shorten <=1.5pt},
  ar/.style={->, line width=0.45pt, shorten >=1.5pt, shorten <=1.5pt},
]
\useasboundingbox (-7.4,-1.2) rectangle (7.4,10.9);
\node[gh] (cl) at (0,10.1)
  {\small classical block calculus\\[1pt]\scriptsize unsigned blocks, involutive $\lnot$};
\node[bx] (bt) at (0,8.0)
  {\small$\bti$\\[1pt]\scriptsize signed blocks, purge};
\node[sm] (td) at (-5.6,8.0)
  {\small$\btd$\\[1pt]\scriptsize depth-reducing};
\node[bx] (g3) at (0,5.9)
  {\small$\Gti$\\[1pt]\scriptsize multiple succedent, G3-style};
\node[bx] (g1) at (0,3.8)
  {\small$\Gip$\\[1pt]\scriptsize single succedent, G3-style};
\node[bx] (g0) at (0,1.7)
  {\small$\Gzi$\\[1pt]\scriptsize G0-style, split contexts};
\node[bx] (ng) at (0,-0.4)
  {\small$\Ngi$\\[1pt]\scriptsize general elimination rules};

\draw[ar] (cl) -- node[lab,right]
  {signs in place of the\\ involution; the purge\\ (\S\ref{sec:block})} (bt);
\draw[eq] (bt) -- node[lab,right]
  {Thm.~\ref{thm:corr}: the blocks\\ are $\Gti$ read upside down} (g3);
\draw[eq] (g3) -- node[lab,right]
  {Prop.~\ref{prop:collapse}: collapse\\ of the succedent} (g1);
\draw[eq] (g1) -- node[lab,right]
  {Thm.~\ref{thm:g0equiv}: equivalence;\\ Thm.~\ref{thm:g0cut}: cut elimination} (g0);
\draw[eq] (g0) -- node[lab,right]
  {Thms.~\ref{thm:ndtosc}, \ref{thm:sctond}: translations;\\
   Thm.~\ref{thm:norm}: normalization} (ng);
\draw[eq] (bt) -- (td);
\node[slab] at (-3.9,7.05)
  {Thm.~\ref{thm:btdadeq}: the same\\ simple blocks;
    Thm.~\ref{thm:term}:\\ termination};
\end{tikzpicture}
\caption{The calculi of the paper and the results that connect them.  A double
arrow is an equivalence of derivability, in the precise sense fixed by the
result that labels it; the single arrow is the passage of
Section~\ref{sec:block}, in which the two features responsible for the
classicality of the block calculus are replaced.  The systems drawn below the
block calculus are calculi of synthesis, in which a derivation is built upward
from initial sequents or from assumptions; $\bti$ and $\btd$ are calculi of
analysis.}
\label{fig:network}
\end{figure}

\section{The language and its Kripke semantics}\label{sec:lang}

\begin{dfn}[Language]\label{def:lang}
Let $\mathit{Var}=\{p,q,r,\dots\}$ be a countable set of propositional
variables.  The formulae of $\mathcal L$ are generated by
\[
A ::= p \mid \bot \mid A\land A \mid A\lor A \mid A\arr A .
\]
We write $A,B,C,D$ for formulae and $\Gamma,\Delta,\Theta$ for finite multisets
of formulae.  Negation is defined: $\lnot A:=A\arr\bot$.  An \emph{atom} is a
propositional variable or $\bot$.  We write $\Sub(A)$ for the set of
subformulae of $A$, and extend the notation to sets and multisets.
\end{dfn}

Taking $\lnot$ as defined rather than primitive shortens every induction by two
cases and costs nothing: the rules that a primitive negation would receive are
derived in Proposition~\ref{prop:negrules} below, and they are exactly the ones
that the reader familiar with the classical calculus would expect.

\begin{dfn}[Kripke models]\label{def:kripke}
A \emph{Kripke model} is a triple $\mathcal M=(W,\le,V)$ where $W$ is a
non-empty set, $\le$ is a preorder on $W$, and $V\colon W\to\mathcal
P(\mathit{Var})$ is monotone, that is, $w\le v$ implies $V(w)\subseteq V(v)$.
Forcing is defined by
\begin{align*}
w&\forces p &&\text{iff}\quad p\in V(w),\\
w&\nforces \bot, \\
w&\forces A\land B &&\text{iff}\quad w\forces A \text{ and } w\forces B,\\
w&\forces A\lor B &&\text{iff}\quad w\forces A \text{ or } w\forces B,\\
w&\forces A\arr B &&\text{iff}\quad \text{for every } v\ge w,\ v\forces A
\text{ implies } v\forces B .
\end{align*}
A sequent $\Gamma\seq\Delta$ is \emph{valid} if for every model and every
$w$, if $w\forces A$ for all $A\in\Gamma$ then $w\forces B$ for some
$B\in\Delta$.
\end{dfn}

\begin{lem}[Monotonicity]\label{lem:mono}
For every formula $A$, every model $\mathcal M$ and all $w\le v$ in $\mathcal
M$: if $w\forces A$ then $v\forces A$.
\end{lem}

\begin{proof}
By induction on the structure of $A$.  For $A=p$ the claim is the monotonicity
of $V$; for $A=\bot$ it is vacuous; for $A=B\land C$ and $A=B\lor C$ it follows
at once from the induction hypothesis.  Let $A=B\arr C$, let $w\forces B\arr C$
and $w\le v$, and let $u\ge v$ with $u\forces B$.  By transitivity $u\ge w$, so
$u\forces C$ by the forcing clause at $w$.  Hence $v\forces B\arr C$.
\end{proof}

Lemma~\ref{lem:mono} is the only semantic fact used in the design of the
calculus, and it is used exactly once: it is what allows a rule to transport
the $\sT$-signed part of a block to a later world, and what forbids it to
transport the $\sF$-signed part.  The asymmetry of the two signs under
monotonicity is, in a precise sense, the whole of the intuitionistic content of
the system.

\section{The block calculus \texorpdfstring{$\bti$}{BTi}}\label{sec:block}

\subsection{Signed blocks and closure}

\begin{dfn}[Signed formulae and blocks]\label{def:block}
A \emph{signed formula} is an expression $\sT A$ or $\sF A$ with $A\in\mathcal
L$.  A \emph{block} $\Pi$ is a finite set of signed formulae.  We write
$\Ptt=\{\sT A: \sT A\in\Pi\}$ for the \emph{$\sT$-part} of $\Pi$, and, for a
multiset $\Gamma$ of formulae, $\sT[\Gamma]=\{\sT A: A\in\Gamma\}$ and
$\sF[\Gamma]=\{\sF A:A\in\Gamma\}$.  The block associated with the sequent
$\Gamma\seq\Delta$ is
\[
\Pi(\Gamma\seq\Delta)\;=\;\sT[\Gamma]\cup\sF[\Delta].
\]
\end{dfn}

\begin{dfn}[Closure]\label{def:closure}
A block $\Pi$ is \emph{closed} if either $\sT p\in\Pi$ and $\sF p\in\Pi$ for
some propositional variable $p$, or $\sT\bot\in\Pi$.  A closed block terminates
its branch and is marked $\times$.  A block that is neither closed nor the
premiss of any applicable rule is \emph{completed open} and is marked $\odot$;
by Definition~\ref{def:rules} this happens exactly when every one of its signed
formulae is $\sT p$, $\sF p$ or $\sF\bot$.
\end{dfn}

Closure is stipulated on atoms only, as in the initial sequents of a G3-style
calculus; that it extends to arbitrary formulae is Lemma~\ref{lem:genclosure}
below.  Observe that $\sF\bot$ never contributes to closure: every world
falsifies $\bot$, so a block containing $\sF\bot$ is not thereby unrealizable.

\subsection{The rules}

\begin{dfn}[The rules of $\bti$]\label{def:rules}
In each rule the principal formula is displayed to the right of the context
$\Pi$; a rule with two children displays them separated by $\mid$.  The rules
fall into four groups.

\smallskip
\noindent\emph{Rules of type $\alpha$} (one child, principal consumed):
\[
\frac{\Pi,\ \sT(A\land B)}{\Pi,\ \sT A,\ \sT B}\ \ \sT\!\land
\qquad\qquad
\frac{\Pi,\ \sF(A\lor B)}{\Pi,\ \sF A,\ \sF B}\ \ \sF\!\lor
\]

\smallskip
\noindent\emph{Rules of type $\beta$} (two children, principal consumed):
\[
\frac{\Pi,\ \sT(A\lor B)}{\Pi,\ \sT A\ \mid\ \Pi,\ \sT B}\ \ \sT\!\lor
\qquad\qquad
\frac{\Pi,\ \sF(A\land B)}{\Pi,\ \sF A\ \mid\ \Pi,\ \sF B}\ \ \sF\!\land
\]

\smallskip
\noindent\emph{Rule with persistence} (two children; the principal formula
survives in the left one):
\[
\frac{\Pi,\ \sT(A\arr B)}{\Pi,\ \sT(A\arr B),\ \sF A\ \ \mid\ \ \Pi,\ \sT B}
\ \ \sT\!\arr
\]

\smallskip
\noindent\emph{Rule of transition} (one child; the $\sF$-part of the context is
purged):
\[
\frac{\Pi,\ \sF(A\arr B)}{\Ptt,\ \sT A,\ \sF B}\ \ \sF\!\arr
\]

\noindent
A \emph{tableau} for a block $\Pi$ is a finite tree rooted at $\Pi$ and grown by
these rules; it is \emph{closed} if every leaf is a closed block.  We say that
$\Pi$ is \emph{refutable} if it admits a closed tableau, and we write
$\bti\vdash_{n}\Pi$ if it admits a closed tableau of height at most $n$.  A
formula $A$ is \emph{provable} if $\{\sF A\}$ is refutable.
\end{dfn}

Three points deserve comment before we proceed.

First, no rule applies to $\sT p$, $\sF p$, $\sF\bot$ or $\sT\bot$: the first
three are inert and the last closes the block.  Every rule therefore has a
compound principal formula, and the calculus is a decomposition procedure in
the strict sense.

Second, the rule $\sT\!\arr$ retains its principal formula in the child that
carries $\sF A$.  Without the retention the calculus would be incomplete: from
$\sT(p\arr q)$ one may need to use the implication again after $\sF p$ has been
further analysed, and the classical device of putting a second copy in the
block is unavailable here because blocks are sets.  The retention is exactly
the repetition of the principal formula in the rule $\arr\!\mathrm L$ of
$\Gip$ and of the multiple-succedent calculi
\cite{Dragalin1988,TroelstraSchwichtenberg2000,NegriVonPlato2001}, and it is
what makes contraction admissible; it is also, in the terminology of the
first-order classical block calculus, what makes $\sT\!\arr$ a rule of type
$\gamma$.

Third, and this is the only place where the calculus is not classical, the rule
$\sF\!\arr$ replaces the context $\Pi$ by its $\sT$-part.  The reason is
semantic and is given in full in Theorem~\ref{thm:sound}: falsifying $A\arr B$
at $w$ requires passing to some $v\ge w$, and only the $\sT$-signed formulae of
$\Pi$ are guaranteed to survive the passage.

\begin{prop}[Derived rules for negation]\label{prop:negrules}
Under the definition $\lnot A:=A\arr\bot$, the following rules are derivable in
$\bti$, in the sense that a block has a closed tableau if and only if it has one
in which every application of $\sT\!\arr$ and $\sF\!\arr$ whose principal
formula is a negation is replaced by an application of the corresponding rule
below:
\[
\frac{\Pi,\ \sT\lnot A}{\Pi,\ \sT\lnot A,\ \sF A}\ \ \sT\!\lnot
\qquad\qquad
\frac{\Pi,\ \sF\lnot A}{\Ptt,\ \sT A}\ \ \sF\!\lnot
\]
\end{prop}

\begin{proof}
For $\sT\!\lnot$: the rule $\sT\!\arr$ with principal $\sT(A\arr\bot)$ produces
the two children $\Pi,\sT\lnot A,\sF A$ and $\Pi,\sT\bot$.  The second is
closed by Definition~\ref{def:closure}, so a closed tableau for the first is
already a closed tableau for the whole; conversely, an application of
$\sT\!\lnot$ is turned into one of $\sT\!\arr$ by adjoining the closed leaf.
For $\sF\!\lnot$: the rule $\sF\!\arr$ with principal $\sF(A\arr\bot)$ produces
$\Ptt,\sT A,\sF\bot$, and $\sF\bot$ is inert, so it may be carried along
without ever being principal or contributing to closure; deleting it from every
block of the subtree yields a closed tableau for $\Ptt,\sT A$, and adjoining it
performs the converse transformation.
\end{proof}

The two derived rules are worth setting beside their classical counterparts.
In the classical block calculus the negation rules are the De Morgan ones and
are of types $\alpha$ and $\beta$; here $\sT\!\lnot$ is a rule with persistence
and $\sF\!\lnot$ is a rule of transition.  The classical involution
$\lnot\lnot A\rightsquigarrow A$ has no counterpart at all: under the sign
discipline, $\sT\lnot\lnot A$ is decomposed by $\sT\!\lnot$ into
$\sF\lnot A$, and $\sF\lnot A$ by $\sF\!\lnot$ into $\sT A$ together with a
purge --- and it is the purge that blocks the classical collapse.

\subsection{Why exactly two rules are anomalous}\label{subsec:anomalous}

Of the six rules of Definition~\ref{def:rules}, four decompose their principal
formula in the way one expects of a tableau, while $\sT\!\arr$ retains it and
$\sF\!\arr$ purges the context.  It would be unsatisfactory to leave these two
departures as stipulations, and they need not be left so: both follow from one
semantic fact about the implication, and they differ because the fact can fail
in two ways.

\begin{dfn}[Local decomposability]\label{def:locdec}
Let $\ast$ be a binary connective and $\mathsf S\in\{\sT,\sF\}$.  The signed
compound $\mathsf S(A\ast B)$ is \emph{locally decomposable} if there are
finitely many sets $S_{1},\dots,S_{k}\subseteq\{\sT A,\sF A,\sT B,\sF B\}$ such
that, in every Kripke model and at every world $w$, the block
$\{\mathsf S(A\ast B)\}$ is realized at $w$ if and only if some $S_{j}$ is
realized at $w$.
\end{dfn}

Local decomposability says that the semantic status of a signed compound at a
world is determined by the status of its immediate subformulae \emph{at that
same world}, and by finitely many alternatives.  It is the exact condition
under which a Smullyan-style rule of type $\alpha$ or $\beta$ is available: the
sets $S_{j}$ are the children.

\begin{prop}\label{prop:locdec}
For all formulae $A$ and $B$, the signed compounds $\sT(A\land B)$,
$\sF(A\land B)$, $\sT(A\lor B)$ and $\sF(A\lor B)$ are locally decomposable.
By contrast, if $A$ and $B$ are distinct propositional variables, neither
$\sT(A\arr B)$ nor $\sF(A\arr B)$ is; a fortiori the implication admits no
decomposition uniform in $A$ and $B$.
\end{prop}

\begin{proof}
The four positive cases are immediate from Definition~\ref{def:kripke}:
$\{\sT(A\land B)\}$ is realized at $w$ exactly when $\{\sT A,\sT B\}$ is;
$\{\sF(A\land B)\}$ exactly when $\{\sF A\}$ or $\{\sF B\}$ is;
$\{\sT(A\lor B)\}$ exactly when $\{\sT A\}$ or $\{\sT B\}$ is; and
$\{\sF(A\lor B)\}$ exactly when $\{\sF A,\sF B\}$ is.  These four equivalences
are the rules $\sT\!\land$, $\sF\!\land$, $\sT\!\lor$, $\sF\!\lor$.

For the negative cases take $A=p$ and $B=q$ distinct propositional variables,
and consider two models.  Let $\mathcal M_{1}$ have the single world $w_{1}$
with $V(w_{1})=\emptyset$, and let $\mathcal M_{2}$ have $w_{2}\le v$ with
$V(w_{2})=\emptyset$ and $V(v)=\{p\}$.  At $w_{1}$ and at $w_{2}$ the same
subsets of $\{\sT p,\sF p,\sT q,\sF q\}$ are realized, namely the subsets of
$\{\sF p,\sF q\}$; so no family $S_{1},\dots,S_{k}$ can distinguish the two
worlds.  The implication itself does distinguish them: $w_{1}\forces p\arr q$,
since no world
of $\mathcal M_{1}$ forces $p$, while $w_{2}\nforces p\arr q$, since
$v\forces p$ and $v\nforces q$.  Hence $\{\sT(p\arr q)\}$ is realized at
$w_{2}$ but not at $w_{1}$ and $\{\sF(p\arr q)\}$ at $w_{1}$ but not at
$w_{2}$, and neither can be locally decomposed.
\end{proof}

The two failures are of different kinds, and the calculus repairs them
differently.

For $\sT(A\arr B)$ the implication from left to right still holds: if
$w\forces A\arr B$ then, by reflexivity, $w\nforces A$ or $w\forces B$.  What
fails is the converse, and it fails because $w\nforces A$ carries no
information about the worlds above $w$.  Retaining the principal formula
repairs exactly this.  Indeed, for every $\Pi$ and every $w$ one has
\[
\Pi\cup\{\sT(A\arr B)\}\ \text{realized at}\ w
\iff
\Pi\cup\{\sT(A\arr B),\sF A\}\ \text{or}\ \Pi\cup\{\sT B\}\ \text{realized at}\ w,
\]
the right-to-left direction using, in its second disjunct, that $w\forces B$
implies $w\forces A\arr B$ by monotonicity.  Delete the retained occurrence and
the equivalence collapses to an implication.  The persistence of $\sT\!\arr$ is
thus not a device imported from the first-order calculus but the precise price
of restoring an equivalence, and it is for the same reason that the rule is
invertible (Lemma~\ref{lem:g3inv}) although its underlying semantic condition
is not.

For $\sF(A\arr B)$ no retention can help, because the falsification is not
witnessed at $w$ at all: it is witnessed at some $v\ge w$, and the only
information that survives the passage is the $\sT$-signed part, by
Lemma~\ref{lem:mono}.  The purge discards exactly what monotonicity does not
license one to keep, and it secures the implication the calculus needs: if
$\Pi\cup\{\sF(A\arr B)\}$ is realized at $w$, then $\Ptt\cup\{\sT A,\sF B\}$ is
realized at some $v\ge w$.  Here, unlike in the previous case, the converse
does \emph{not} hold --- the $\sF$-signed formulae of $\Pi$ need not be
falsified at $v$, and nothing can make them so --- and the residue of the repair
is exactly the failure of invertibility recorded in
Remark~\ref{rem:noinvR} below, which $\sF\!\arr$ alone among the rules of $\bti$
suffers.

Proposition~\ref{prop:locdec} also serves as a criterion of applicability.  A
connective may be added to the block format precisely when both of its signed
forms are locally decomposable; when one of them is not, one must ask whether
the failure is of the first kind, curable by retention, or of the second,
curable only by a transition.  Under the classical, single-world semantics the
question never arises, every signed compound being decomposable in the
corresponding sense; and that is why the propositional part of the classical
block calculus consists of rules of type $\alpha$ and $\beta$ alone.

\subsection{Analyticity, weakening, generalized closure}

\begin{dfn}[Signed subformulae]\label{def:signedsub}
For a block $\Pi$ put $\Sub(\Pi)=\bigcup\{\Sub(A) : \sT A\in\Pi \text{ or }
\sF A\in\Pi\}$ and
$\Subpm(\Pi)=\{\sT A,\sF A : A\in\Sub(\Pi)\}$.
\end{dfn}

\begin{prop}[Subformula property]\label{prop:subf}
Every signed formula occurring in a tableau for $\Pi$ belongs to
$\Subpm(\Pi)$.  Consequently every block occurring in a tableau for $\Pi$ is a
subset of the finite set $\Subpm(\Pi)$, and there are at most
$2^{|\Subpm(\Pi)|}$ distinct such blocks.
\end{prop}

\begin{proof}
By induction on the depth of the node.  At the root the claim is trivial.  If a
node carries $\Pi'\subseteq\Subpm(\Pi)$, inspection of the six rules shows
that every signed formula of a child is either a signed formula of $\Pi'$ or a
signed immediate subformula of the principal formula of $\Pi'$; in either case
it lies in $\Subpm(\Pi)$.  The rule $\sF\!\arr$ only removes formulae in
addition, which cannot take a block out of $\Subpm(\Pi)$.  The counting is
immediate, $\Subpm(\Pi)$ being finite.
\end{proof}

\begin{lem}[Height-preserving weakening]\label{lem:weak}
If $\bti\vdash_{n}\Pi$ and $\Pi\subseteq\Pi'$, then $\bti\vdash_{n}\Pi'$.
\end{lem}

\begin{proof}
By induction on $n$.  If $n=0$ then $\Pi$ is closed, and since closure is a
condition on membership, every superset of $\Pi$ is closed as well.  Let
$n>0$ and let the root step of a closed tableau for $\Pi$ apply a rule $R$ with
principal formula $\varphi\in\Pi$ and children $\Pi_{1}$ (and possibly
$\Pi_{2}$), each with $\bti\vdash_{n-1}\Pi_{j}$.  Put $\Delta=\Pi'\setminus\Pi$
and apply $R$ to $\Pi'$ with the same principal formula, which is legitimate
since $\varphi\in\Pi\subseteq\Pi'$.  If $R$ is not $\sF\!\arr$, the resulting
children are $\Pi_{j}\cup(\Delta\setminus\{\varphi\})\supseteq\Pi_{j}$; if $R$
is $\sF\!\arr$, the resulting child is
$(\Pi')^{\sT}\cup\{\sT A,\sF B\}
=\Ptt\cup\Delta^{\sT}\cup\{\sT A,\sF B\}\supseteq\Pi_{1}$.  In both cases the
induction hypothesis applies to each child and yields a closed tableau of height
at most $n-1$ for it, whence $\bti\vdash_{n}\Pi'$.
\end{proof}

\begin{lem}[Generalized closure]\label{lem:genclosure}
For every formula $A$ and every block $\Pi$, the block $\Pi\cup\{\sT A,\sF A\}$
is refutable.
\end{lem}

\begin{proof}
By Lemma~\ref{lem:weak} it suffices to treat $\Pi=\emptyset$, that is, to refute
$\{\sT A,\sF A\}$; we argue by induction on the number of connectives of $A$.
If $A$ is a variable or $\bot$ the block is closed.  If $A=B\land C$, apply
$\sT\!\land$, obtaining $\{\sT B,\sT C,\sF(B\land C)\}$, and then $\sF\!\land$,
whose two children are $\{\sT B,\sT C,\sF B\}$ and $\{\sT B,\sT C,\sF C\}$;
each is refutable by the induction hypothesis and Lemma~\ref{lem:weak}.  If
$A=B\lor C$, apply $\sF\!\lor$ and then $\sT\!\lor$, and conclude in the same
way.  If $A=B\arr C$, apply $\sF\!\arr$ first: the child is
$\{\sT(B\arr C),\sT B,\sF C\}$, the principal formula being the only
$\sF$-signed formula purged.  Now apply $\sT\!\arr$: the children are
$\{\sT(B\arr C),\sT B,\sF C,\sF B\}$, refutable by the induction hypothesis for
$B$, and $\{\sT B,\sF C,\sT C\}$, refutable by the induction hypothesis for
$C$.  In every case the induction hypothesis is applied to a formula with
strictly fewer connectives.
\end{proof}

The order of the two applications in the last case is not an accident, and it
is the first symptom of a phenomenon that will recur: the transition rule must
be applied \emph{before} the rule with persistence, exactly as, in the
first-order classical calculus, a rule of type $\delta$ must precede the rules
of type $\gamma$ that use the constant it introduces.

\subsection{Absorption of the structural rules}

The calculus has no structural rules, and it is worth recording precisely how
each of them is accounted for.  Contraction is absorbed by
Definition~\ref{def:block}: a block is a set, so a signed formula is present or
absent and carries no multiplicity.  For the implication, where a set-theoretic
reading would not suffice --- the principal formula of $\arr\!\mathrm L$ must
remain available --- contraction is absorbed a second time, by the persistence
built into $\sT\!\arr$.  Weakening is absorbed by
Definition~\ref{def:closure} together with Lemma~\ref{lem:weak}: a block closes
as soon as it contains a complementary atomic pair, whatever else it contains,
so that adding formulae can only make closure easier.  Cut is absent because
every formula produced by a rule is a signed immediate subformula of the
principal formula, as Proposition~\ref{prop:subf} records.  All three
absorptions are visible on the face of the rules; none of them is a theorem
about the calculus proved after the fact.

\subsection{Worked examples}

We display tableaux with the root at the top, each node carrying its entire
block, the applied rule to the right of the line, and terminal blocks in braces
marked $\times$ or $\odot$; a bifurcation shows its two children separated by
$\mid$, and a branch too wide for the page is continued as a separate tree
rooted at the block from which it issues.

\begin{exa}[A provable formula]\label{ex:k}
The formula $p\arr(q\arr p)$ is provable.  We refute $\{\sF(p\arr(q\arr p))\}$:
\begin{prooftree}
\AxiomC{$\sF(p\arr(q\arr p))$}
\RightLabel{\ $\sF\!\arr$}
\UnaryInfC{$\sT p,\ \sF(q\arr p)$}
\RightLabel{\ $\sF\!\arr$}
\UnaryInfC{$\{\sT p,\ \sT q,\ \sF p\}\ \times$}
\end{prooftree}
The second transition purges $\sF(q\arr p)$ --- the principal formula itself,
being $\sF$-signed --- while retaining $\sT p$, and the resulting block closes
on the pair $\sT p,\sF p$.
\end{exa}

\begin{exa}[Transitivity of the implication]\label{ex:trans}
The formula $(p\arr q)\arr((q\arr r)\arr(p\arr r))$ is provable.  Three
transitions bring the root block to a block with three $\sT$-signed formulae:
\begin{prooftree}
\AxiomC{$\sF((p\arr q)\arr((q\arr r)\arr(p\arr r)))$}
\RightLabel{\ $\sF\!\arr$}
\UnaryInfC{$\sT(p\arr q),\ \sF((q\arr r)\arr(p\arr r))$}
\RightLabel{\ $\sF\!\arr$}
\UnaryInfC{$\sT(p\arr q),\ \sT(q\arr r),\ \sF(p\arr r)$}
\RightLabel{\ $\sF\!\arr$}
\UnaryInfC{$\sT(p\arr q),\ \sT(q\arr r),\ \sT p,\ \sF r$}
\end{prooftree}
Each transition purges the $\sF$-signed formula of the block, which on this
branch is always the principal one, so nothing is in fact lost.  Two
applications of $\sT\!\arr$ now close the tableau:
\begin{prooftree}
\AxiomC{$\sT(p\arr q),\ \sT(q\arr r),\ \sT p,\ \sF r$}
\RightLabel{\ $\sT\!\arr$}
\UnaryInfC{$\{\sT(p\arr q),\sT(q\arr r),\sT p,\sF r,\sF p\}\ \times
\ \ \mid\ \ \sT(q\arr r),\ \sT p,\ \sF r,\ \sT q$}
\end{prooftree}
\begin{prooftree}
\AxiomC{$\sT(q\arr r),\ \sT p,\ \sF r,\ \sT q$}
\RightLabel{\ $\sT\!\arr$}
\UnaryInfC{$\{\sT(q\arr r),\sT p,\sF r,\sT q,\sF q\}\ \times
\ \ \mid\ \ \{\sT p,\sF r,\sT q,\sT r\}\ \times$}
\end{prooftree}
The left leaf of the first bifurcation closes on $\sT p,\sF p$, the left leaf of
the second on $\sT q,\sF q$, and the right leaf of the second on $\sT r,\sF r$.
Note that the principal formula retained in each left leaf plays no further
part: retention is not needed here.  Example~\ref{ex:dnlem} exhibits a case in
which it is.
\end{exa}

\begin{exa}[Excluded middle]\label{ex:lem}
The formula $p\lor\lnot p$ is not provable.  We attempt to refute
$\{\sF(p\lor\lnot p)\}$:
\begin{prooftree}
\AxiomC{$\sF(p\lor\lnot p)$}
\RightLabel{\ $\sF\!\lor$}
\UnaryInfC{$\sF p,\ \sF\lnot p$}
\RightLabel{\ $\sF\!\lnot$}
\UnaryInfC{$\{\sT p\}\ \odot$}
\end{prooftree}
The transition purges $\sF p$, and the resulting block is completed open: no
rule applies to $\sT p$.  Reading the branch as a two-element model --- a root
$w$ at which nothing is forced, and a successor $v$ at which $p$ is forced ---
one has $w\nforces p$ and, since $v\forces p$, also $w\nforces\lnot p$; hence
$w\nforces p\lor\lnot p$.  In the classical block calculus the corresponding
tableau closes, because there the context is transported entire.
\end{exa}

\begin{exa}[Peirce's law]\label{ex:peirce}
The formula $((p\arr q)\arr p)\arr p$ is not provable.  We attempt to refute
$\{\sF(((p\arr q)\arr p)\arr p)\}$.  The first step is a transition:
\begin{prooftree}
\AxiomC{$\sF(((p\arr q)\arr p)\arr p)$}
\RightLabel{\ $\sF\!\arr$}
\UnaryInfC{$\sT((p\arr q)\arr p),\ \sF p$}
\RightLabel{\ $\sT\!\arr$}
\UnaryInfC{$\sT((p\arr q)\arr p),\ \sF p,\ \sF(p\arr q)
\ \ \mid\ \ \{\sT p,\ \sF p\}\ \times$}
\end{prooftree}
The right child closes.  On the left child a further transition is applied,
with principal formula $\sF(p\arr q)$:
\begin{prooftree}
\AxiomC{$\sT((p\arr q)\arr p),\ \sF p,\ \sF(p\arr q)$}
\RightLabel{\ $\sF\!\arr$}
\UnaryInfC{$\sT((p\arr q)\arr p),\ \sT p,\ \sF q$}
\RightLabel{\ $\sT\!\arr$}
\UnaryInfC{$\sT((p\arr q)\arr p),\ \sT p,\ \sF q,\ \sF(p\arr q)
\ \ \mid\ \ \{\sT p,\ \sF q\}\ \odot$}
\end{prooftree}
The purge has removed $\sF p$, which in the classical calculus would have
closed the branch against $\sT p$.  The right child is completed open, and the
tableau cannot be closed.  The open branch yields the countermodel
$\mathcal M=(\{w_{0},w_{1}\},\le,V)$ with $w_{0}\le w_{1}$, $V(w_{0})=\emptyset$
and $V(w_{1})=\{p\}$.  One checks directly that $w_{1}\nforces p\arr q$, since
$w_{1}\forces p$ and $w_{1}\nforces q$, and that $w_{0}\nforces p\arr q$ for the
same reason; hence $w_{0}\forces(p\arr q)\arr p$ vacuously, while
$w_{0}\nforces p$, so that $w_{0}\nforces((p\arr q)\arr p)\arr p$.
\end{exa}

\begin{exa}[The double negation of excluded middle]\label{ex:dnlem}
Although $p\lor\lnot p$ is not provable, its double negation is, and the
refutation shows the retention of $\sT\!\lnot$ doing work that nothing else
could do.  Write $E$ for $p\lor\lnot p$ and recall the derived rules of
Proposition~\ref{prop:negrules}.
\begin{prooftree}
\AxiomC{$\sF\lnot\lnot E$}
\RightLabel{\ $\sF\!\lnot$}
\UnaryInfC{$\sT\lnot E$}
\RightLabel{\ $\sT\!\lnot$}
\UnaryInfC{$\sT\lnot E,\ \sF(p\lor\lnot p)$}
\RightLabel{\ $\sF\!\lor$}
\UnaryInfC{$\sT\lnot E,\ \sF p,\ \sF\lnot p$}
\RightLabel{\ $\sF\!\lnot$}
\UnaryInfC{$\sT\lnot E,\ \sT p$}
\RightLabel{\ $\sT\!\lnot$}
\UnaryInfC{$\sT\lnot E,\ \sT p,\ \sF(p\lor\lnot p)$}
\RightLabel{\ $\sF\!\lor$}
\UnaryInfC{$\{\sT\lnot E,\ \sT p,\ \sF p,\ \sF\lnot p\}\ \times$}
\end{prooftree}
The branch does not bifurcate once.  Its fourth step is a transition: it
introduces $\sT p$ and purges the two $\sF$-signed formulae of the block, among
them the $\sF p$ that $\sT p$ would have contradicted.  Closure is thereby
postponed, and what makes it possible all the same is that $\sT\lnot E$ survives
the purge, being $\sT$-signed, and is still available to be decomposed a second
time; the second decomposition reinstates $\sF(p\lor\lnot p)$, now in a block
that carries $\sT p$ as well, and one further application of $\sF\!\lor$
produces the complementary pair.  Delete the retention from $\sT\!\lnot$ and
$\sT\lnot E$ would be consumed at the second step; the branch would then run
$\sF(p\lor\lnot p)$, $\{\sF p,\sF\lnot p\}$, $\{\sT p\}$ and stop, completed
open.
\end{exa}

Example~\ref{ex:peirce} repays a second look, because it isolates the working
of the purge.  The blocks of the classical refutation of Peirce's law and those
of the present attempt agree step for step, up to the sign discipline, until
the second transition; there the classical calculus carries $\lnot p$ forward
and closes, whereas here $\sF p$ is deleted and the branch survives.  A single
deletion separates the two logics.
\section{The sequent calculus \texorpdfstring{$\Gti$}{G3Ti}}\label{sec:g3}

A block carries, in general, several $\sF$-signed formulae, and the rule
$\sF\!\arr$ deletes all of them but the one it decomposes.  The sequent
calculus that matches this behaviour is therefore not $\LJ$ but the
multiple-succedent intuitionistic calculus of Maehara \cite{Maehara1954} and
Dragalin \cite{Dragalin1988}, in the structural-rule-free form given to it by
Troelstra and Schwichtenberg \cite{TroelstraSchwichtenberg2000} and by Negri and
von Plato \cite{NegriVonPlato2001}.  We fix it here in the notation of the
present paper.

\subsection{The calculus}

\begin{dfn}[The calculus $\Gti$]\label{def:g3}
A sequent is an expression $\Gamma\seq\Delta$ with $\Gamma,\Delta$ finite
multisets of formulae.  The calculus $\Gti$ has the initial sequents
\[
p,\Gamma\seq\Delta,p\ \ (\mathrm{id})
\qquad\qquad
\bot,\Gamma\seq\Delta\ \ (\mathrm{L}\bot)
\]
with $p$ a propositional variable, and the rules
\[
\frac{A,B,\Gamma\seq\Delta}{A\land B,\Gamma\seq\Delta}\ \land\mathrm L
\qquad
\frac{A,\Gamma\seq\Delta\quad B,\Gamma\seq\Delta}
     {A\lor B,\Gamma\seq\Delta}\ \lor\mathrm L
\qquad
\frac{A\arr B,\Gamma\seq\Delta,A\quad B,\Gamma\seq\Delta}
     {A\arr B,\Gamma\seq\Delta}\ \arr\mathrm L
\]
\[
\frac{\Gamma\seq\Delta,A\quad\Gamma\seq\Delta,B}
     {\Gamma\seq\Delta,A\land B}\ \land\mathrm R
\qquad
\frac{\Gamma\seq\Delta,A,B}{\Gamma\seq\Delta,A\lor B}\ \lor\mathrm R
\qquad
\frac{A,\Gamma\seq B}{\Gamma\seq\Delta,A\arr B}\ \arr\mathrm R
\]
We write $\Gti\vdash_{n}\Gamma\seq\Delta$ if the sequent has a derivation of
height at most $n$, the height of a derivation being the length of its longest
branch.  A rule is \emph{admissible} if the derivability of its premisses
entails that of its conclusion, \emph{height-preserving admissible} (hp-admissible)
if moreover no increase in height is required, and \emph{invertible} if the
derivability of its conclusion entails that of each of its premisses.
\end{dfn}

Two features of $\Gti$ single it out among sequent calculi for intuitionistic
logic.  The rule $\arr\!\mathrm L$ repeats its principal formula in the left
premiss; this is Dragalin's device, and it is what makes contraction
hp-admissible.  The rule $\arr\!\mathrm R$ discards the context $\Delta$ in
passing to its premiss; this is the succedent-emptying that the purge of the
block calculus transcribes, and it is the only rule of $\Gti$ that is not
invertible.  Everything peculiar to the intuitionistic case is a consequence of
one or the other.

\subsection{Weakening and initial sequents}

\begin{lem}[Height-preserving weakening]\label{lem:g3wk}
The rules
\[
\frac{\Gamma\seq\Delta}{A,\Gamma\seq\Delta}\ \wkL
\qquad\qquad
\frac{\Gamma\seq\Delta}{\Gamma\seq\Delta,A}\ \wkR
\]
are hp-admissible in $\Gti$.
\end{lem}

\begin{proof}
By induction on the height $n$ of the given derivation, for both rules
simultaneously.  If $n=0$ the sequent is initial, and the initial sequents
admit arbitrary contexts on both sides, so the weakened sequent is initial as
well.  Let $n>0$ and let $R$ be the last rule.  If $R$ is one of
$\land\mathrm L$, $\lor\mathrm L$, $\arr\mathrm L$, $\land\mathrm R$,
$\lor\mathrm R$, then $R$ carries its context unchanged from conclusion to
premisses, so the induction hypothesis applied to the premisses and a
re-application of $R$ give the conclusion at the same height.  If $R$ is
$\arr\mathrm R$, with conclusion $\Gamma\seq\Delta,C\arr D$ and premiss
$C,\Gamma\seq D$, then for $\wkR$ nothing is to be done, since $\arr\mathrm R$
allows an arbitrary succedent context: one re-applies it with $\Delta,A$ in
place of $\Delta$.  For $\wkL$ one applies the induction hypothesis to the
premiss, obtaining $A,C,\Gamma\seq D$, and re-applies $\arr\mathrm R$.
\end{proof}

\begin{lem}[Generalized initial sequents]\label{lem:g3id}
For every formula $A$ and all $\Gamma,\Delta$, the sequent
$A,\Gamma\seq\Delta,A$ is derivable in $\Gti$.
\end{lem}

\begin{proof}
By induction on the number of connectives of $A$.  If $A$ is a variable the
sequent is $(\mathrm{id})$, and if $A=\bot$ it is $(\mathrm{L}\bot)$.

Let $A=B\land C$.  By $\land\mathrm R$ it suffices to derive
$B\land C,\Gamma\seq\Delta,B$ and $B\land C,\Gamma\seq\Delta,C$; each follows
by $\land\mathrm L$ from $B,C,\Gamma\seq\Delta,B$, respectively
$B,C,\Gamma\seq\Delta,C$, which are available by the induction hypothesis, the
formulae $B$ and $C$ having fewer connectives than $A$.

Let $A=B\lor C$.  By $\lor\mathrm L$ it suffices to derive
$B,\Gamma\seq\Delta,B\lor C$ and $C,\Gamma\seq\Delta,B\lor C$; each follows by
$\lor\mathrm R$ from $B,\Gamma\seq\Delta,B,C$, respectively
$C,\Gamma\seq\Delta,B,C$, available by the induction hypothesis.

Let $A=B\arr C$.  By $\arr\mathrm R$ it suffices to derive
$B,B\arr C,\Gamma\seq C$.  Apply $\arr\mathrm L$ to the antecedent occurrence
of $B\arr C$: the premisses are $B\arr C,B,\Gamma\seq C,B$ and
$C,B,\Gamma\seq C$, both instances of the induction hypothesis for $B$ and for
$C$.
\end{proof}

\subsection{Invertibility, contraction, thinning}

\begin{lem}[Invertibility]\label{lem:g3inv}
The following are hp-admissible in $\Gti$.
\begin{enumerate}
\item[\textup{(i)}] From $A\land B,\Gamma\seq\Delta$ infer $A,B,\Gamma\seq\Delta$.
\item[\textup{(ii)}] From $A\lor B,\Gamma\seq\Delta$ infer $A,\Gamma\seq\Delta$,
and likewise $B,\Gamma\seq\Delta$.
\item[\textup{(iii)}] From $A\arr B,\Gamma\seq\Delta$ infer $B,\Gamma\seq\Delta$.
\item[\textup{(iv)}] From $\Gamma\seq\Delta,A\land B$ infer $\Gamma\seq\Delta,A$,
and likewise $\Gamma\seq\Delta,B$.
\item[\textup{(v)}] From $\Gamma\seq\Delta,A\lor B$ infer $\Gamma\seq\Delta,A,B$.
\end{enumerate}
The left premiss of $\arr\mathrm L$ is recovered from its conclusion by
Lemma~\ref{lem:g3wk}.  The rule $\arr\mathrm R$ is not invertible.
\end{lem}

\begin{proof}
Each item is proved by induction on the height $n$ of the derivation of the
given sequent, and the five inductions have the same shape; we give (iii) and
(v) in full and indicate the remaining ones.

\emph{(iii).}  If $n=0$ the sequent $A\arr B,\Gamma\seq\Delta$ is initial.  If
it is $(\mathrm{id})$, the atomic pair witnessing it cannot involve the
compound formula $A\arr B$, so it lies in $\Gamma$ and $\Delta$ and
$B,\Gamma\seq\Delta$ is again $(\mathrm{id})$; if it is $(\mathrm{L}\bot)$ then
$\bot\in\Gamma$, since $A\arr B\ne\bot$, and $B,\Gamma\seq\Delta$ is again
$(\mathrm{L}\bot)$.  Let $n>0$.  If the last rule is $\arr\mathrm L$ with the
displayed occurrence of $A\arr B$ as principal formula, its right premiss is
$B,\Gamma\seq\Delta$, derivable at height $n-1$.  If the last rule $R$ is any
other rule with a principal formula distinct from the displayed occurrence,
then $A\arr B$ occurs in the context of each premiss, unless $R$ is
$\arr\mathrm R$; in the former case the induction hypothesis applied to the
premisses and a re-application of $R$ give the conclusion at height at most
$n$.  If $R$ is $\arr\mathrm R$, with conclusion
$A\arr B,\Gamma\seq\Delta_{0},C\arr D$ and premiss $C,A\arr B,\Gamma\seq D$,
the induction hypothesis gives $C,B,\Gamma\seq D$ at height at most $n-1$, and
$\arr\mathrm R$ yields $B,\Gamma\seq\Delta_{0},C\arr D$.

\emph{(v).}  If $n=0$, the sequent $\Gamma\seq\Delta,A\lor B$ is initial and,
$A\lor B$ being compound, so is $\Gamma\seq\Delta,A,B$.  Let $n>0$.  If the
last rule is $\lor\mathrm R$ with $A\lor B$ principal, its premiss is
$\Gamma\seq\Delta,A,B$.  If the last rule is $\arr\mathrm R$, with premiss
$C,\Gamma\seq D$, then $\arr\mathrm R$ applied to the same premiss with the
succedent context $\Delta,A,B$ yields $\Gamma\seq\Delta,A,B,C\arr D$ at the
same height.  In every other case the displayed occurrence of $A\lor B$ is
passive, and the induction hypothesis applied to the premisses followed by a
re-application of the rule gives the claim.

\emph{(i), (ii), (iv).}  The argument is the same in all three cases.  In the
base case one uses that the principal formula is compound, and hence can be
neither the atom witnessing $(\mathrm{id})$ nor $\bot$.  In the inductive step
one distinguishes whether the displayed formula is principal in the last rule,
in which case the required sequent is a premiss, or passive, in which case the
induction hypothesis applies to the premisses.  The case of $\arr\mathrm R$ is
handled as above: that rule imposes no constraint on the succedent context of
its conclusion, and none at all on a passive antecedent formula.
\end{proof}

\begin{rem}\label{rem:noinvR}
That $\arr\mathrm R$ is not invertible is seen at once: $p\seq p,q\arr r$ is an
instance of $(\mathrm{id})$, while $q,p\seq r$ is not derivable, as the
one-element model with $V(w)=\{p,q\}$ shows.  By Theorem~\ref{thm:corr} the
same holds of the block rule $\sF\!\arr$, and of it alone among the rules of
$\bti$.
\end{rem}

\begin{lem}[Height-preserving contraction]\label{lem:g3ctr}
The rules
\[
\frac{A,A,\Gamma\seq\Delta}{A,\Gamma\seq\Delta}\ \ctrL
\qquad\qquad
\frac{\Gamma\seq\Delta,A,A}{\Gamma\seq\Delta,A}\ \ctrR
\]
are hp-admissible in $\Gti$.
\end{lem}

\begin{proof}
Simultaneous induction on the height $n$ of the given derivation.  For $n=0$
the sequent is initial and so is its contractum, since deleting one of two
identical formulae cannot destroy an atomic pair nor an occurrence of $\bot$ in
the antecedent.  Let $n>0$ with last rule $R$.

If neither displayed occurrence is principal in $R$, the induction hypothesis
applies to the premisses of $R$ --- in each of which both occurrences are still
present, the rules carrying their contexts unchanged --- and $R$ is re-applied.
The rule $\arr\mathrm R$ deserves a word in the case of $\ctrR$: if its
conclusion is $\Gamma\seq\Delta,A,A$ with $A$ passive, then $A\ne C\arr D$ for
the principal $C\arr D$, and one re-applies $\arr\mathrm R$ to the same premiss
with succedent context $\Delta,A$.

Suppose one displayed occurrence is principal in $R$.  We treat the four cases
in which the induction hypothesis alone does not suffice.

\emph{$A=C\land D$ contracted on the left, $R=\land\mathrm L$.}  The premiss is
$C,D,C\land D,\Gamma\seq\Delta$, of height $n-1$.  By
Lemma~\ref{lem:g3inv}(i), applied to the remaining occurrence of $C\land D$,
we obtain $C,D,C,D,\Gamma\seq\Delta$ at height at most $n-1$; two applications
of the induction hypothesis give $C,D,\Gamma\seq\Delta$, and $\land\mathrm L$
concludes $C\land D,\Gamma\seq\Delta$ at height at most $n$.

\emph{$A=C\lor D$ contracted on the left, $R=\lor\mathrm L$.}  The premisses are
$C,C\lor D,\Gamma\seq\Delta$ and $D,C\lor D,\Gamma\seq\Delta$.  By
Lemma~\ref{lem:g3inv}(ii) the first yields $C,C,\Gamma\seq\Delta$ and the
second $D,D,\Gamma\seq\Delta$, both at height at most $n-1$; the induction
hypothesis and $\lor\mathrm L$ conclude.

\emph{$A=C\arr D$ contracted on the left, $R=\arr\mathrm L$.}  The premisses are
$C\arr D,C\arr D,\Gamma\seq\Delta,C$ and $D,C\arr D,\Gamma\seq\Delta$, both at
height at most $n-1$.  The induction hypothesis applied to the first gives
$C\arr D,\Gamma\seq\Delta,C$.  To the second we apply
Lemma~\ref{lem:g3inv}(iii), obtaining $D,D,\Gamma\seq\Delta$, and then the
induction hypothesis, obtaining $D,\Gamma\seq\Delta$.  One application of
$\arr\mathrm L$ concludes $C\arr D,\Gamma\seq\Delta$ at height at most $n$.

\emph{Contraction on the right.}  If $A=C\land D$ and $R=\land\mathrm R$, the
premisses are $\Gamma\seq\Delta,C\land D,C$ and $\Gamma\seq\Delta,C\land D,D$;
Lemma~\ref{lem:g3inv}(iv) applied to the remaining occurrence turns the first
into $\Gamma\seq\Delta,C,C$ and the second into $\Gamma\seq\Delta,D,D$, and the
induction hypothesis followed by $\land\mathrm R$ concludes.  If $A=C\lor D$
and $R=\lor\mathrm R$, the premiss is $\Gamma\seq\Delta,C\lor D,C,D$;
Lemma~\ref{lem:g3inv}(v) gives $\Gamma\seq\Delta,C,D,C,D$, two applications of
the induction hypothesis give $\Gamma\seq\Delta,C,D$, and $\lor\mathrm R$
concludes.  If $A=C\arr D$ and $R=\arr\mathrm R$, the conclusion is
$\Gamma\seq\Delta,C\arr D,C\arr D$ and the premiss is $C,\Gamma\seq D$; a
single application of $\arr\mathrm R$ to that premiss, with succedent context
$\Delta$, gives $\Gamma\seq\Delta,C\arr D$ at height at most $n$.
\end{proof}

\begin{lem}[Thinning of $\bot$ on the right]\label{lem:botthin}
If $\Gti\vdash_{n}\Gamma\seq\Delta,\bot$ then $\Gti\vdash_{n}\Gamma\seq\Delta$.
\end{lem}

\begin{proof}
By induction on $n$.  If the sequent is $(\mathrm{id})$, the witnessing atom is
a propositional variable and hence lies in $\Delta$, so $\Gamma\seq\Delta$ is
again $(\mathrm{id})$; if it is $(\mathrm{L}\bot)$, then $\bot\in\Gamma$ and
$\Gamma\seq\Delta$ is again $(\mathrm{L}\bot)$.  In the inductive step the
displayed $\bot$ is never principal, no rule having $\bot$ as principal formula
on the right; one applies the induction hypothesis to the premisses and
re-applies the rule, the case of $\arr\mathrm R$ being handled by choosing
$\Delta$ as succedent context.
\end{proof}

\subsection{The block calculus is \texorpdfstring{$\Gti$}{G3Ti} read upside down}

\begin{con}\label{con:blockseq}
For a block $\Pi$ we write $\Gamma_{\Pi}=\{A:\sT A\in\Pi\}$ and
$\Delta_{\Pi}=\{A:\sF A\in\Pi\}$, and we read these finite sets as multisets in
which every formula occurs once.
\end{con}

\begin{thm}[Correspondence]\label{thm:corr}
For every block $\Pi$: $\Pi$ is refutable in $\bti$ if and only if
$\Gti\vdash\Gamma_{\Pi}\seq\Delta_{\Pi}$.  Equivalently, for all finite
multisets $\Gamma,\Delta$ the sequent $\Gamma\seq\Delta$ is derivable in $\Gti$
if and only if the block $\Pi(\Gamma\seq\Delta)=\sT[\Gamma]\cup\sF[\Delta]$ is
refutable in $\bti$.
\end{thm}

\begin{proof}
\emph{From sequents to blocks.}  We show that if $\Gti\vdash_{n}\Gamma\seq
\Delta$ then $\Pi=\sT[\Gamma]\cup\sF[\Delta]$ is refutable, arguing by
induction on the pair $(n,m)$ ordered lexicographically, where $m$ is the total
number of formula occurrences in $\Gamma$ and $\Delta$.  The auxiliary
parameter $m$ is needed for the following reason: a block is a set, so the
block rule that consumes its principal formula removes its only copy, whereas
the corresponding sequent rule removes one occurrence among possibly several.
The discrepancy is eliminated at the outset of each step.

Suppose first that $\Gamma$ or $\Delta$ contains a repeated formula.  By
Lemma~\ref{lem:g3ctr} the sequent obtained by deleting one of the two
occurrences is derivable with height at most $n$, and it has $m-1$ occurrences;
since the associated block is unchanged, the induction hypothesis applies and
gives the claim.  We may therefore assume that $\Gamma$ and $\Delta$ are
repetition-free, so that the principal formula of the last rule occurs exactly
once on its side.

If $\Gamma\seq\Delta$ is $(\mathrm{id})$ then $\sT p,\sF p\in\Pi$ and $\Pi$ is
closed; if it is $(\mathrm{L}\bot)$ then $\sT\bot\in\Pi$ and $\Pi$ is closed.
For the inductive step we treat each rule, writing $\Gamma=\Gamma_{0},A\ast B$
or $\Delta=\Delta_{0},A\ast B$ as the case requires, with $A\ast B$ not
occurring in $\Gamma_{0}$, respectively $\Delta_{0}$.  In each case the block
of the premiss is reached from $\Pi$ by exactly one application of the
corresponding block rule, and the induction hypothesis applies to the premiss,
whose derivation has height at most $n-1$.

If the last rule is $\land\mathrm L$, then $\Pi=\sT[\Gamma_{0}]\cup\{\sT(A\land
B)\}\cup\sF[\Delta]$, and one application of $\sT\!\land$ produces the child
$\sT[\Gamma_{0}]\cup\{\sT A,\sT B\}\cup\sF[\Delta]$, which is the block of the
premiss and is refutable by the induction hypothesis.  If the last rule is
$\lor\mathrm L$, one application of $\sT\!\lor$ produces the two blocks of the
two premisses.  If the last rule is $\arr\mathrm L$, one application of
$\sT\!\arr$ produces the left child
$\sT[\Gamma_{0}]\cup\{\sT(A\arr B)\}\cup\sF[\Delta]\cup\{\sF A\}$, the block of
the left premiss $A\arr B,\Gamma_{0}\seq\Delta,A$, and the right child
$\sT[\Gamma_{0}]\cup\{\sT B\}\cup\sF[\Delta]$, the block of the right premiss.
Dually, $\land\mathrm R$ is matched by $\sF\!\land$ and $\lor\mathrm R$ by
$\sF\!\lor$.  Finally, if the last rule is $\arr\mathrm R$, with conclusion
$\Gamma\seq\Delta_{0},A\arr B$ and premiss $A,\Gamma\seq B$, then one
application of $\sF\!\arr$ to
$\Pi=\sT[\Gamma]\cup\sF[\Delta_{0}]\cup\{\sF(A\arr B)\}$ produces
$\Ptt\cup\{\sT A,\sF B\}=\sT[\Gamma]\cup\{\sT A,\sF B\}$, which is precisely
the block of the premiss.  In each case the induction hypothesis applies to the
blocks of the premisses.

\emph{From blocks to sequents.}  We show by induction on the height of a closed
tableau for $\Pi$ that $\Gti\vdash\Gamma_{\Pi}\seq\Delta_{\Pi}$.  If $\Pi$ is
closed because $\sT p,\sF p\in\Pi$, then $p\in\Gamma_{\Pi}$ and
$p\in\Delta_{\Pi}$, so the sequent is $(\mathrm{id})$; if because
$\sT\bot\in\Pi$, it is $(\mathrm{L}\bot)$.  Otherwise let $R$ be the rule
applied at the root.

Suppose $R=\sT\!\land$ with principal $\sT(A\land B)$, so that the child is
$\Pi_{1}=(\Pi\setminus\{\sT(A\land B)\})\cup\{\sT A,\sT B\}$.  By the induction
hypothesis $\Gti\vdash\Gamma_{\Pi_{1}}\seq\Delta_{\Pi_{1}}$, that is,
$(\Gamma_{\Pi}\setminus\{A\land B\})\cup\{A,B\}\seq\Delta_{\Pi}$ with the
antecedent read as a multiset without repetitions.  By Lemma~\ref{lem:g3wk} we
may adjoin further copies, obtaining
$A,B,\Gamma_{\Pi}\setminus\{A\land B\}\seq\Delta_{\Pi}$, and $\land\mathrm L$
yields $A\land B,\Gamma_{\Pi}\setminus\{A\land B\}\seq\Delta_{\Pi}$, which is
$\Gamma_{\Pi}\seq\Delta_{\Pi}$.  The rules $\sT\!\lor$, $\sF\!\land$,
$\sF\!\lor$ are treated in the same way, using $\lor\mathrm L$,
$\land\mathrm R$, $\lor\mathrm R$ respectively.

Suppose $R=\sT\!\arr$ with principal $\sT(A\arr B)$.  The induction hypothesis
applied to the left child $\Pi\cup\{\sF A\}$ gives
$\Gamma_{\Pi}\seq\Delta_{\Pi},A$ up to the adjunction of copies, and applied to
the right child $(\Pi\setminus\{\sT(A\arr B)\})\cup\{\sT B\}$ gives
$B,\Gamma_{\Pi}\setminus\{A\arr B\}\seq\Delta_{\Pi}$.  Since $A\arr B$ belongs
to $\Gamma_{\Pi}$, the first of these is the left premiss of $\arr\mathrm L$
and the second its right premiss, so $\arr\mathrm L$ yields
$\Gamma_{\Pi}\seq\Delta_{\Pi}$.

Suppose finally $R=\sF\!\arr$ with principal $\sF(A\arr B)$, so that the child
is $\Ptt\cup\{\sT A,\sF B\}$.  The induction hypothesis gives
$A,\Gamma_{\Pi}\seq B$ up to the adjunction of copies, and $\arr\mathrm R$
yields $\Gamma_{\Pi}\seq\Delta_{\Pi}\setminus\{A\arr B\},A\arr B$, that is,
$\Gamma_{\Pi}\seq\Delta_{\Pi}$; note that $\arr\mathrm R$ permits an arbitrary
succedent context, so the formulae of $\Delta_{\Pi}$ purged by $\sF\!\arr$ are
restored at exactly this point.
\end{proof}

The last observation is worth isolating.  The purge does not lose information:
what the block calculus deletes on passing to the child, the rule
$\arr\!\mathrm R$ reinstates on passing from the premiss to the conclusion.  In
the refutational reading the deleted formulae are simply of no further use,
since the countermodel under construction has moved to a later world; in the
synthetic reading they are context that the rule is free to restore.

\begin{rem}[The two directions are not equally direct]\label{rem:algo}
The proof just given is algorithmic in both directions, but the algorithms are
not of the same kind, and the asymmetry is worth recording, as it is of the sort
that von Plato \cite{vonPlato2017} isolates in the translation between tree and
linear natural deduction.  From a closed tableau to a derivation one proceeds by
a single pass: invert the orientation of the tree, read each closed leaf as an
initial sequent and each rule application as the corresponding sequent rule, and
insert weakenings where the passage from a set to a multiset requires copies.
From a derivation to a closed tableau a preliminary pass is needed, in which
repetitions are contracted away; without it the block rule, which removes the
only copy of its principal formula, would delete material that the sequent rule
leaves in place.  The bookkeeping that one format spends on discharge labels the
other spends on multiplicities.
\end{rem}

\begin{cor}[Subformula property for $\Gti$]\label{cor:subfg3}
Every formula occurring in a derivation of $\Gamma\seq\Delta$ in $\Gti$ is a
subformula of a formula of $\Gamma\cup\Delta$.
\end{cor}

\begin{proof}
Immediate by inspection of the rules, each of which has, in its premisses, only
subformulae of formulae of its conclusion; alternatively, transport
Proposition~\ref{prop:subf} through Theorem~\ref{thm:corr}.
\end{proof}

\begin{rem}[Why cut does not permute in $\Gti$]\label{rem:cutobstruction}
It is natural to expect, at this point, a syntactic proof of the admissibility
of
\[
\frac{\Gamma\seq\Delta,A\qquad A,\Gamma'\seq\Delta'}
     {\Gamma,\Gamma'\seq\Delta,\Delta'}\ \mathrm{Cut}
\]
by the usual double induction on the weight of the cut formula and the sum of
the heights of the premisses.  The reductions in which the cut formula is
principal in both premisses go through, and so do the permutations upward past
every rule except one.  The exception is the permutation past $\arr\mathrm R$
in the \emph{right} premiss.  Suppose $A,\Gamma'\seq\Delta'$ is concluded by
$\arr\mathrm R$ with principal formula $C\arr D\in\Delta'$, from the premiss
$C,A,\Gamma'\seq D$.  Permuting the cut upward yields
$C,\Gamma,\Gamma'\seq\Delta,D$, whose succedent is not the single formula $D$;
and $\arr\mathrm R$ cannot be re-applied to it, since that rule requires its
premiss to have exactly one formula in the succedent.  The obstruction
disappears when $\Delta$ is empty, which is why the single-succedent calculus
$\Gip$ introduced in Section~\ref{sec:sem} admits the standard syntactic
proof.  We shall
obtain the admissibility of cut in $\Gti$ semantically, as
Corollary~\ref{cor:semcut} below; it is a fact about the calculus, not about the
method, that the succedent-emptying rule and the naive permutation of cut do not
sit well together.
\end{rem}

\section{Soundness, completeness and their corollaries}\label{sec:sem}

\subsection{Realization and soundness}

\begin{dfn}[Realization]\label{def:real}
Let $\mathcal M=(W,\le,V)$ be a Kripke model and $w\in W$.  The block $\Pi$ is
\emph{realized at $w$} if $w\forces A$ for every $\sT A\in\Pi$ and
$w\nforces B$ for every $\sF B\in\Pi$.  It is \emph{realizable} if it is
realized at some world of some model.
\end{dfn}

Realizability is the exact counterpart, for blocks, of the invalidity of the
associated sequent: $\Pi(\Gamma\seq\Delta)$ is realizable if and only if
$\Gamma\seq\Delta$ is not valid, by Definitions~\ref{def:kripke} and
\ref{def:real}.

\begin{thm}[Soundness]\label{thm:sound}
If a block $\Pi$ is refutable in $\bti$, then $\Pi$ is not realizable.
\end{thm}

\begin{proof}
By induction on the height $n$ of a closed tableau for $\Pi$.

If $n=0$ then $\Pi$ is closed.  Suppose $\Pi$ were realized at $w$.  If
$\sT p,\sF p\in\Pi$ then $w\forces p$ and $w\nforces p$, which is impossible;
if $\sT\bot\in\Pi$ then $w\forces\bot$, again impossible.

Let $n>0$ and let the root apply a rule $R$ with principal formula $\varphi$.
We show that if $\Pi$ is realized at $w$ in $\mathcal M$, then some child of
$\Pi$ under $R$ is realized at some world of $\mathcal M$; the induction
hypothesis, applicable since each child has a closed tableau of height at most
$n-1$, then yields the contradiction.

\emph{$R=\sT\!\land$, $\varphi=\sT(A\land B)$.}  From $w\forces A\land B$ we get
$w\forces A$ and $w\forces B$, so the child, which differs from $\Pi$ by the
removal of $\varphi$ and the addition of $\sT A,\sT B$, is realized at $w$.

\emph{$R=\sF\!\lor$, $\varphi=\sF(A\lor B)$.}  From $w\nforces A\lor B$ we get
$w\nforces A$ and $w\nforces B$; the child is realized at $w$.

\emph{$R=\sT\!\lor$, $\varphi=\sT(A\lor B)$.}  From $w\forces A\lor B$ we get
$w\forces A$ or $w\forces B$, so one of the two children is realized at $w$.

\emph{$R=\sF\!\land$, $\varphi=\sF(A\land B)$.}  From $w\nforces A\land B$ we
get $w\nforces A$ or $w\nforces B$, so one of the two children is realized at
$w$.

\emph{$R=\sT\!\arr$, $\varphi=\sT(A\arr B)$.}  Suppose $w\forces A\arr B$.  If
$w\nforces A$, the left child $\Pi\cup\{\sF A\}$ --- in which $\varphi$ is
retained --- is realized at $w$.  If $w\forces A$, then, $\le$ being reflexive,
the forcing clause for the implication gives $w\forces B$, and the right child
$(\Pi\setminus\{\varphi\})\cup\{\sT B\}$ is realized at $w$.

\emph{$R=\sF\!\arr$, $\varphi=\sF(A\arr B)$.}  Suppose $w\nforces A\arr B$.  By
the forcing clause there is $v\ge w$ with $v\forces A$ and $v\nforces B$.  By
Lemma~\ref{lem:mono}, $v\forces C$ for every $\sT C\in\Pi$, so $\Ptt$ is
realized at $v$; adding $\sT A$ and $\sF B$, the child
$\Ptt\cup\{\sT A,\sF B\}$ is realized at $v$.  Nothing is claimed about the
$\sF$-signed formulae of $\Pi$ at $v$, and nothing can be: this is exactly why
they are purged.
\end{proof}

\begin{cor}[Soundness of $\Gti$]\label{cor:soundg3}
If $\Gti\vdash\Gamma\seq\Delta$ then $\Gamma\seq\Delta$ is valid.
\end{cor}

\begin{proof}
By Theorem~\ref{thm:corr} the block $\Pi(\Gamma\seq\Delta)$ is refutable, hence
by Theorem~\ref{thm:sound} not realizable, hence $\Gamma\seq\Delta$ is valid.
\end{proof}

\subsection{Saturation and completeness}

Completeness is proved on the sequent side, where the structural lemmas of
Section~\ref{sec:g3} are available in the form in which they are needed, and
transported back to blocks by Theorem~\ref{thm:corr}.  Throughout this
subsection $\Sigma$ denotes a finite set of formulae closed under subformulae.

\begin{dfn}[Saturated pairs]\label{def:sat}
A pair $(\Gamma,\Delta)$ of subsets of $\Sigma$ is \emph{$\Sigma$-saturated} if
\begin{enumerate}
\item[\textup{(a)}] $\Gamma\seq\Delta$ is not derivable in $\Gti$;
\item[\textup{(b)}] $A\land B\in\Gamma$ implies $A\in\Gamma$ and $B\in\Gamma$;
\item[\textup{(c)}] $A\lor B\in\Gamma$ implies $A\in\Gamma$ or $B\in\Gamma$;
\item[\textup{(d)}] $A\arr B\in\Gamma$ implies $A\in\Delta$ or $B\in\Gamma$;
\item[\textup{(e)}] $A\land B\in\Delta$ implies $A\in\Delta$ or $B\in\Delta$;
\item[\textup{(f)}] $A\lor B\in\Delta$ implies $A\in\Delta$ and $B\in\Delta$.
\end{enumerate}
\end{dfn}

Clause (a) already yields $\bot\notin\Gamma$, since $\bot,\Gamma\seq\Delta$ is
an initial sequent.  Observe that no clause governs an implication in $\Delta$:
that case is not a condition on the pair but a demand for a further world, and
it is met in the truth lemma by producing one.

\begin{lem}[Saturation]\label{lem:sat}
Let $\Gamma,\Delta\subseteq\Sigma$ with $\Gamma\seq\Delta$ not derivable in
$\Gti$.  Then there are $\Gamma^{*},\Delta^{*}$ with
$\Gamma\subseteq\Gamma^{*}\subseteq\Sigma$,
$\Delta\subseteq\Delta^{*}\subseteq\Sigma$ and $(\Gamma^{*},\Delta^{*})$
$\Sigma$-saturated.
\end{lem}

\begin{proof}
We enlarge the pair step by step, each step adding one formula of $\Sigma$ to
$\Gamma$ or to $\Delta$ and preserving the non-derivability of the associated
sequent.  Since $\Sigma$ is finite and $|\Gamma|+|\Delta|$ strictly increases,
the process halts, and it halts only at a pair satisfying (b)--(f).  It remains
to check that each violated clause admits a licit step.  In each case we assume
that the pair $(\Gamma,\Delta)$ has $\Gamma\seq\Delta$ underivable, and we
argue by contradiction.

\emph{(b).}  Let $A\land B\in\Gamma$ and $A\notin\Gamma$, and suppose that
$A,\Gamma\seq\Delta$ is derivable.  Writing $\Gamma=A\land B,\Gamma_{0}$ and
applying Lemma~\ref{lem:g3inv}(i) to the occurrence of $A\land B$, we obtain
$A,A,B,\Gamma_{0}\seq\Delta$; one application of Lemma~\ref{lem:g3ctr} gives
$A,B,\Gamma_{0}\seq\Delta$, and $\land\mathrm L$ gives $\Gamma\seq\Delta$,
against the hypothesis.  Hence $A$ may be added to $\Gamma$; the same argument
applies to $B$.

\emph{(c).}  Let $A\lor B\in\Gamma$ with $A\notin\Gamma$ and $B\notin\Gamma$,
and suppose that both $A,\Gamma\seq\Delta$ and $B,\Gamma\seq\Delta$ are
derivable.  Writing $\Gamma=A\lor B,\Gamma_{0}$ and applying
Lemma~\ref{lem:g3inv}(ii) to each, we obtain $A,A,\Gamma_{0}\seq\Delta$ and
$B,B,\Gamma_{0}\seq\Delta$; contraction gives $A,\Gamma_{0}\seq\Delta$ and
$B,\Gamma_{0}\seq\Delta$, and $\lor\mathrm L$ gives $\Gamma\seq\Delta$.  Hence
at least one of the two additions is licit.

\emph{(d).}  Let $A\arr B\in\Gamma$ with $A\notin\Delta$ and $B\notin\Gamma$,
and suppose that both $\Gamma\seq\Delta,A$ and $B,\Gamma\seq\Delta$ are
derivable.  Write $\Gamma=A\arr B,\Gamma_{0}$.  The first sequent is the left
premiss of $\arr\mathrm L$.  From the second, Lemma~\ref{lem:g3inv}(iii) gives
$B,B,\Gamma_{0}\seq\Delta$ and contraction gives $B,\Gamma_{0}\seq\Delta$, the
right premiss.  Hence $\arr\mathrm L$ gives $\Gamma\seq\Delta$.

\emph{(e).}  Let $A\land B\in\Delta$ with $A\notin\Delta$ and $B\notin\Delta$,
and suppose that both $\Gamma\seq\Delta,A$ and $\Gamma\seq\Delta,B$ are
derivable.  Write $\Delta=\Delta_{0},A\land B$.  Applying
Lemma~\ref{lem:g3inv}(iv) to the occurrence of $A\land B$ in the first gives
$\Gamma\seq\Delta_{0},A,A$, and contraction gives $\Gamma\seq\Delta_{0},A$;
similarly $\Gamma\seq\Delta_{0},B$.  Then $\land\mathrm R$ gives
$\Gamma\seq\Delta$.

\emph{(f).}  Let $A\lor B\in\Delta$ with $A\notin\Delta$, and suppose that
$\Gamma\seq\Delta,A$ is derivable.  Write $\Delta=\Delta_{0},A\lor B$.
Lemma~\ref{lem:g3inv}(v) gives $\Gamma\seq\Delta_{0},A,B,A$, contraction gives
$\Gamma\seq\Delta_{0},A,B$, and $\lor\mathrm R$ gives $\Gamma\seq\Delta$.  The
same for $B$.

All the formulae adjoined are subformulae of formulae already present, hence
members of $\Sigma$, and the pair remains a pair of subsets of $\Sigma$.
\end{proof}

\begin{dfn}[The canonical model]\label{def:canon}
Let $\mathcal M_{\Sigma}=(W_{\Sigma},\le,V)$ where $W_{\Sigma}$ is the set of
$\Sigma$-saturated pairs, $(\Gamma,\Delta)\le(\Gamma',\Delta')$ if and only if
$\Gamma\subseteq\Gamma'$, and $V(\Gamma,\Delta)=\Gamma\cap\mathit{Var}$.
\end{dfn}

The relation $\le$ is reflexive and transitive, and $V$ is monotone by
construction, so $\mathcal M_{\Sigma}$ is a Kripke model whenever
$W_{\Sigma}\ne\emptyset$.  It has at most $4^{|\Sigma|}$ worlds.

\begin{lem}[Truth lemma]\label{lem:truth}
For every $A\in\Sigma$ and every $(\Gamma,\Delta)\in W_{\Sigma}$:
\begin{enumerate}
\item[\textup{(1)}] if $A\in\Gamma$ then $(\Gamma,\Delta)\forces A$;
\item[\textup{(2)}] if $A\in\Delta$ then $(\Gamma,\Delta)\nforces A$.
\end{enumerate}
\end{lem}

\begin{proof}
By induction on the number of connectives of $A$, the two claims being proved
simultaneously.

\emph{$A=p$.}  (1) is the definition of $V$.  For (2), if $p\in\Delta$ and
$p\in\Gamma$ then $\Gamma\seq\Delta$ would be $(\mathrm{id})$, against clause
(a); hence $p\notin\Gamma$ and $(\Gamma,\Delta)\nforces p$.

\emph{$A=\bot$.}  (1) is vacuous, since $\bot\notin\Gamma$ by clause (a).  (2)
holds because no world forces $\bot$.

\emph{$A=B\land C$.}  For (1), clause (b) gives $B,C\in\Gamma$, and the
induction hypothesis gives $\forces B$ and $\forces C$.  For (2), clause (e)
gives $B\in\Delta$ or $C\in\Delta$, and the induction hypothesis gives
$\nforces B$ or $\nforces C$.

\emph{$A=B\lor C$.}  Symmetrically, using clauses (c) and (f).

\emph{$A=B\arr C$, claim (1).}  Let $B\arr C\in\Gamma$ and let
$(\Gamma',\Delta')\ge(\Gamma,\Delta)$ with $(\Gamma',\Delta')\forces B$.  Since
$\Gamma\subseteq\Gamma'$ we have $B\arr C\in\Gamma'$, so clause (d) applied to
$(\Gamma',\Delta')$ gives $B\in\Delta'$ or $C\in\Gamma'$.  The first is
impossible: the induction hypothesis (2) for $B$ would give
$(\Gamma',\Delta')\nforces B$.  Hence $C\in\Gamma'$ and, by the induction
hypothesis (1) for $C$, $(\Gamma',\Delta')\forces C$.  As $(\Gamma',\Delta')$
was arbitrary, $(\Gamma,\Delta)\forces B\arr C$.

\emph{$A=B\arr C$, claim (2).}  Let $B\arr C\in\Delta$.  We claim first that
$B,\Gamma\seq C$ is not derivable in $\Gti$.  Indeed, were it derivable, one
application of $\arr\mathrm R$ --- which imposes no constraint on the succedent
context of its conclusion --- would give
$\Gamma\seq\Delta\setminus\{B\arr C\},B\arr C$, that is $\Gamma\seq\Delta$,
against clause (a).  Now $B\arr C\in\Sigma$ and $\Sigma$ is closed under
subformulae, so $\{B\}\cup\Gamma\subseteq\Sigma$ and $\{C\}\subseteq\Sigma$.  By
Lemma~\ref{lem:sat} there is a $\Sigma$-saturated pair $(\Gamma',\Delta')$ with
$\{B\}\cup\Gamma\subseteq\Gamma'$ and $C\in\Delta'$.  Then
$(\Gamma,\Delta)\le(\Gamma',\Delta')$, and the induction hypothesis gives
$(\Gamma',\Delta')\forces B$ and $(\Gamma',\Delta')\nforces C$.  Hence
$(\Gamma,\Delta)\nforces B\arr C$.
\end{proof}

\begin{thm}[Completeness, with the finite model property]\label{thm:complete}
Let $\Gamma_{0}\seq\Delta_{0}$ be a sequent not derivable in $\Gti$.  Then
there is a finite Kripke model with at most $4^{|\Sigma|}$ worlds, where
$\Sigma=\Sub(\Gamma_{0}\cup\Delta_{0})$, and a world of it at which every
formula of $\Gamma_{0}$ is forced and no formula of $\Delta_{0}$ is forced.
Consequently $\Gamma_{0}\seq\Delta_{0}$ is derivable in $\Gti$ if and only if
it is valid.
\end{thm}

\begin{proof}
The set $\Sigma$ is finite and closed under subformulae, and
$\Gamma_{0},\Delta_{0}\subseteq\Sigma$.  By Lemma~\ref{lem:sat} there is a
$\Sigma$-saturated pair $(\Gamma^{*},\Delta^{*})$ with
$\Gamma_{0}\subseteq\Gamma^{*}$ and $\Delta_{0}\subseteq\Delta^{*}$; in
particular $W_{\Sigma}\ne\emptyset$ and $\mathcal M_{\Sigma}$ is a model.  By
Lemma~\ref{lem:truth}, $(\Gamma^{*},\Delta^{*})\forces A$ for every
$A\in\Gamma_{0}$ and $(\Gamma^{*},\Delta^{*})\nforces B$ for every
$B\in\Delta_{0}$.  Hence $\Gamma_{0}\seq\Delta_{0}$ is not valid.  The converse
implication is Corollary~\ref{cor:soundg3}.
\end{proof}

\begin{cor}[Completeness of $\bti$]\label{cor:completebt}
A block $\Pi$ is refutable in $\bti$ if and only if it is not realizable; and if
it is realizable, it is realized in a finite model with at most
$4^{|\Sub(\Pi)|}$ worlds.
\end{cor}

\begin{proof}
Combine Theorems~\ref{thm:sound}, \ref{thm:corr} and \ref{thm:complete},
recalling that $\Pi$ is realizable if and only if
$\Gamma_{\Pi}\seq\Delta_{\Pi}$ is not valid.
\end{proof}

\begin{cor}[Decidability]\label{cor:dec}
Refutability in $\bti$, and hence derivability in $\Gti$ and provability in
intuitionistic propositional logic, is decidable.
\end{cor}

\begin{proof}
Given $\Pi$, put $\Sigma=\Sub(\Pi)$ and let $k=4^{|\Sigma|}$.  By
Theorem~\ref{thm:sound} refutability implies non-realizability, and by
Corollary~\ref{cor:completebt} non-refutability implies realizability in a model
with at most $k$ worlds and with $V$ restricted to the finitely many variables
occurring in $\Sigma$.  There are finitely many such models up to isomorphism,
and realizability in a given finite model is decidable by evaluating the
finitely many forcing conditions.  Enumerating them therefore decides
refutability.
\end{proof}

The bound is of course extravagant; Section~\ref{sec:term} replaces this
enumeration by a terminating refutation search.

\begin{cor}[Admissibility of cut in $\Gti$]\label{cor:semcut}
The rule
\[
\frac{\Gamma\seq\Delta,A\qquad A,\Gamma'\seq\Delta'}
     {\Gamma,\Gamma'\seq\Delta,\Delta'}\ \mathrm{Cut}
\]
is admissible in $\Gti$.
\end{cor}

\begin{proof}
Assume both premisses derivable, hence valid by Corollary~\ref{cor:soundg3}.
Let $w$ force every formula of $\Gamma,\Gamma'$.  By the validity of the first
premiss, either $w$ forces some formula of $\Delta$, and we are done, or
$w\forces A$; in the latter case $w$ forces every formula of $A,\Gamma'$, so by
the validity of the second premiss $w$ forces some formula of $\Delta'$.  Hence
$\Gamma,\Gamma'\seq\Delta,\Delta'$ is valid, and by
Theorem~\ref{thm:complete} it is derivable.
\end{proof}

\subsection{The disjunction property}

The disjunction property is not available in the classical calculus, and it is
worth obtaining it here by an argument internal to the block presentation,
where it turns on a single structural fact: a block whose $\sT$-part is empty
closes only because one of its members closes on its own.

\begin{lem}[Blocks without $\sT$-formulae]\label{lem:allF}
Let $\Pi$ be a block with $\Ptt=\emptyset$.  If $\Pi$ is refutable in $\bti$,
then $\{\varphi\}$ is refutable for some $\varphi\in\Pi$.
\end{lem}

\begin{proof}
By induction on the height $n$ of a closed tableau for $\Pi$.  The case $n=0$
does not arise: closure requires either a pair $\sT p,\sF p$ or the formula
$\sT\bot$, and neither is available in a block with empty $\sT$-part.  So
$n\ge1$, and the root applies a rule $R$ whose principal formula $\varphi_{0}$
is, again for want of $\sT$-formulae, one of $\sF(A\land B)$, $\sF(A\lor B)$,
$\sF(A\arr B)$.

\emph{$R=\sF\!\lor$, $\varphi_{0}=\sF(A\lor B)$.}  The child
$\Pi_{1}=(\Pi\setminus\{\varphi_{0}\})\cup\{\sF A,\sF B\}$ has empty
$\sT$-part and is refutable at height at most $n-1$, so by the induction
hypothesis $\{\psi\}$ is refutable for some $\psi\in\Pi_{1}$.  If
$\psi\in\Pi\setminus\{\varphi_{0}\}$ we are done.  Otherwise $\psi$ is $\sF A$
or $\sF B$; applying $\sF\!\lor$ to $\{\varphi_{0}\}$ produces the single child
$\{\sF A,\sF B\}$, which is refutable by Lemma~\ref{lem:weak}, so
$\{\varphi_{0}\}$ is refutable.

\emph{$R=\sF\!\land$, $\varphi_{0}=\sF(A\land B)$.}  The two children
$(\Pi\setminus\{\varphi_{0}\})\cup\{\sF A\}$ and
$(\Pi\setminus\{\varphi_{0}\})\cup\{\sF B\}$ have empty $\sT$-part and are
refutable at height at most $n-1$.  By the induction hypothesis each contains a
refutable singleton.  If either of the two witnesses lies in
$\Pi\setminus\{\varphi_{0}\}$ we are done.  Otherwise the witnesses are
$\sF A$ and $\sF B$ respectively, and applying $\sF\!\land$ to
$\{\varphi_{0}\}$ produces the children $\{\sF A\}$ and $\{\sF B\}$, both
refutable; hence $\{\varphi_{0}\}$ is refutable.

\emph{$R=\sF\!\arr$, $\varphi_{0}=\sF(A\arr B)$.}  The child is
$\Ptt\cup\{\sT A,\sF B\}=\{\sT A,\sF B\}$, refutable at height at most $n-1$.
The same rule applied to $\{\varphi_{0}\}$ produces exactly the same child, so
$\{\varphi_{0}\}$ is refutable.  Here the purge does the work: the remaining
$\sF$-formulae of $\Pi$ play no part whatsoever in the refutation.
\end{proof}

\begin{thm}[Disjunction property]\label{thm:dp}
If $A\lor B$ is provable in $\bti$, then $A$ is provable or $B$ is provable.
\end{thm}

\begin{proof}
Let $\{\sF(A\lor B)\}$ be refutable.  It is not closed, so a closed tableau for
it applies a rule at the root; the only signed formula available is
$\sF(A\lor B)$, so the rule is $\sF\!\lor$ and the child $\{\sF A,\sF B\}$ is
refutable.  This block has empty $\sT$-part, so by Lemma~\ref{lem:allF} either
$\{\sF A\}$ or $\{\sF B\}$ is refutable.
\end{proof}

\subsection{From multiple to single succedents}

The bridge to natural deduction runs through a single-succedent calculus, and
we record here the equivalence that licenses the passage.  We write $\Gip$ for
the single-succedent G3-style calculus for intuitionistic propositional logic
of \cite[\S3.5]{TroelstraSchwichtenberg2000} and \cite[Ch.~5]{NegriVonPlato2001}:
its sequents are $\Gamma\seq C$ with $C$ a formula, its initial sequents are
$p,\Gamma\seq p$ and $\bot,\Gamma\seq C$, and its rules are
\[
\frac{A,B,\Gamma\seq C}{A\land B,\Gamma\seq C}
\quad
\frac{A,\Gamma\seq C\ \ B,\Gamma\seq C}{A\lor B,\Gamma\seq C}
\quad
\frac{A\arr B,\Gamma\seq A\ \ B,\Gamma\seq C}{A\arr B,\Gamma\seq C}
\]
\[
\frac{\Gamma\seq A\ \ \Gamma\seq B}{\Gamma\seq A\land B}
\quad
\frac{\Gamma\seq A}{\Gamma\seq A\lor B}
\quad
\frac{\Gamma\seq B}{\Gamma\seq A\lor B}
\quad
\frac{A,\Gamma\seq B}{\Gamma\seq A\arr B}
\]
For $\Gip$ weakening and contraction are hp-admissible, the rules other than
$\arr\mathrm R$ and $\lor\mathrm R_{i}$ are invertible, cut is admissible, and
the calculus is sound and complete for Kripke semantics; all of this is
standard and we shall use it without further comment
\cite[3.5.5--3.5.11, 4.1.9]{TroelstraSchwichtenberg2000},
\cite[Ch.~5]{NegriVonPlato2001}, \cite{DyckhoffNegri2000}.

\begin{prop}[Collapse]\label{prop:collapse}
Write $\bigvee\Delta$ for the disjunction of the formulae of $\Delta$ in a fixed
association, with $\bigvee\emptyset=\bot$.  For all finite multisets
$\Gamma,\Delta$,
\[
\Gti\vdash\Gamma\seq\Delta
\qquad\text{if and only if}\qquad
\Gip\vdash\Gamma\seq\textstyle\bigvee\Delta .
\]
In particular the two calculi prove the same sequents with a single formula in
the succedent, and $\Gti$ is a conservative extension of $\Gip$.
\end{prop}

\begin{proof}
By the forcing clause for disjunction, a world forces some formula of $\Delta$
if and only if it forces $\bigvee\Delta$; hence $\Gamma\seq\Delta$ is valid if
and only if $\Gamma\seq\bigvee\Delta$ is valid.  The claim follows from the
soundness and completeness of $\Gti$ (Corollary~\ref{cor:soundg3} and
Theorem~\ref{thm:complete}) and of $\Gip$.
\end{proof}

\begin{cor}\label{cor:blockprov}
A formula $A$ is provable in $\bti$ if and only if $\Gip\vdash{}\seq A$, that
is, if and only if $A$ is a theorem of intuitionistic propositional logic.
\end{cor}

\begin{proof}
By Theorem~\ref{thm:corr}, $\{\sF A\}$ is refutable if and only if
$\Gti\vdash{}\seq A$, and by Proposition~\ref{prop:collapse} this holds if and
only if $\Gip\vdash{}\seq A$.
\end{proof}

\section{A G0-style calculus and natural deduction}\label{sec:nd}

We now carry out, for the intuitionistic block calculus, the programme that
Kamide and Negri \cite{KamideNegri2025} carry out for the extended
Belnap--Dunn and intuitionistic logics and that was applied to the classical
block calculus in \cite{CuconatoKJM}: the G3-style calculus is transformed into
a G0-style one, with independent contexts, explicit structural rules and
generalized initial sequents, and the latter is put in correspondence with a
natural deduction system whose elimination rules are general in the sense of
von Plato \cite{vonPlato2001} and Schroeder-Heister
\cite{SchroederHeister1984}.  The technique of bidirectional translation between
a G0-style sequent calculus and general-elimination natural deduction is that of
\cite{NegriVonPlatoJSL2001}, and it goes back, for the relation between
normalization and cut elimination, to Gentzen's own remarks
\cite{Gentzen1935}.

\subsection{The calculus \texorpdfstring{$\Gzi$}{G0Ti}}

\begin{dfn}[The calculus $\Gzi$]\label{def:g0}
Sequents are of the form $\Gamma\seq C$ with $\Gamma$ a finite multiset and $C$
a formula.  The calculus $\Gzi$ has the generalized initial sequents
\[
A\seq A\ \ (\mathrm{id}^{s})
\qquad\qquad
\bot\seq C\ \ (\mathrm{L}\bot^{s})
\]
for arbitrary formulae $A$ and $C$, the structural rules
\[
\frac{\Gamma\seq C}{A,\Gamma\seq C}\ \wkL
\qquad\qquad
\frac{A,A,\Gamma\seq C}{A,\Gamma\seq C}\ \ctrL
\]
and the operational rules
\[
\frac{A,B,\Gamma\seq C}{A\land B,\Gamma\seq C}\ \land\mathrm L
\qquad
\frac{A,\Gamma\seq C\quad B,\Delta\seq C}{A\lor B,\Gamma,\Delta\seq C}\ \lor\mathrm L^{s}
\qquad
\frac{\Gamma\seq A\quad B,\Delta\seq C}{A\arr B,\Gamma,\Delta\seq C}\ \arr\mathrm L^{s}
\]
\[
\frac{\Gamma\seq A\quad\Delta\seq B}{\Gamma,\Delta\seq A\land B}\ \land\mathrm R^{s}
\qquad
\frac{\Gamma\seq A}{\Gamma\seq A\lor B}\ \lor\mathrm R_{1}
\qquad
\frac{\Gamma\seq B}{\Gamma\seq A\lor B}\ \lor\mathrm R_{2}
\qquad
\frac{A,\Gamma\seq B}{\Gamma\seq A\arr B}\ \arr\mathrm R
\]
\end{dfn}

Two differences from $\Gti$ are essential to what follows.  The contexts of the
two-premiss rules are independent, which is what makes them match the way the
general elimination rules of natural deduction combine their subderivations;
and $\arr\mathrm L^{s}$ does not repeat its principal formula, contraction
being now a primitive rule rather than an absorbed one.  A third difference is
invisible: since the succedent is a single formula, $\arr\mathrm R$ discards
nothing, so the purge of the block calculus leaves no trace at this level.  We
return to the point in Section~\ref{sec:comp}.

\begin{thm}[Equivalence of $\Gzi$ and $\Gip$]\label{thm:g0equiv}
For every sequent $\Gamma\seq C$,
\[
\Gzi\vdash\Gamma\seq C
\qquad\text{if and only if}\qquad
\Gip\vdash\Gamma\seq C .
\]
\end{thm}

\begin{proof}
\emph{From $\Gip$ to $\Gzi$.}  Each initial sequent and rule of $\Gip$ is
derivable in $\Gzi$.  The initial sequent $p,\Gamma\seq p$ follows from
$(\mathrm{id}^{s})$ by $\wkL$, and $\bot,\Gamma\seq C$ from
$(\mathrm{L}\bot^{s})$ by $\wkL$.  The rules $\land\mathrm L$,
$\lor\mathrm R_{i}$ and $\arr\mathrm R$ are shared verbatim.  For the
context-sharing two-premiss rules one applies the corresponding
independent-context rule and contracts the duplicated context: from
$A,\Gamma\seq C$ and $B,\Gamma\seq C$, the rule $\lor\mathrm L^{s}$ gives
$A\lor B,\Gamma,\Gamma\seq C$ and repeated $\ctrL$ gives $A\lor B,\Gamma\seq C$;
from $\Gamma\seq A$ and $\Gamma\seq B$, the rule $\land\mathrm R^{s}$ and
$\ctrL$ give $\Gamma\seq A\land B$.  For $\arr\mathrm L$ of $\Gip$, whose
premisses are $A\arr B,\Gamma\seq A$ and $B,\Gamma\seq C$, one applies
$\arr\mathrm L^{s}$ to obtain $A\arr B,A\arr B,\Gamma,\Gamma\seq C$ and then
contracts, both on $A\arr B$ and on the formulae of $\Gamma$.

\emph{From $\Gzi$ to $\Gip$.}  Each initial sequent and rule of $\Gzi$ is
derivable or admissible in $\Gip$.  The generalized initial sequents are
derivable in $\Gip$ by the standard induction on the number of connectives
\cite[3.5.4]{TroelstraSchwichtenberg2000}, and $\bot\seq C$ is initial.  The
rules $\wkL$ and $\ctrL$ are hp-admissible in $\Gip$
\cite[3.5.5, 3.5.9]{TroelstraSchwichtenberg2000}.  The independent-context
rules are obtained from the context-sharing ones by prior weakening: from
$A,\Gamma\seq C$ and $B,\Delta\seq C$, weakening yields $A,\Gamma,\Delta\seq C$
and $B,\Gamma,\Delta\seq C$, and $\lor\mathrm L$ gives
$A\lor B,\Gamma,\Delta\seq C$; from $\Gamma\seq A$ and $B,\Delta\seq C$,
weakening yields $A\arr B,\Gamma,\Delta\seq A$ and $B,\Gamma,\Delta\seq C$, and
$\arr\mathrm L$ gives $A\arr B,\Gamma,\Delta\seq C$; the case of
$\land\mathrm R^{s}$ is the same.  The one-premiss rules are shared.
\end{proof}

\begin{thm}[Cut elimination for $\Gzi$]\label{thm:g0cut}
The rule
\[
\frac{\Gamma\seq A\qquad A,\Delta\seq C}{\Gamma,\Delta\seq C}\ \mathrm{Cut}
\]
is admissible in $\Gzi$.
\end{thm}

\begin{proof}
Let $\Gzi\vdash\Gamma\seq A$ and $\Gzi\vdash A,\Delta\seq C$.  By
Theorem~\ref{thm:g0equiv}, $\Gip\vdash\Gamma\seq A$ and
$\Gip\vdash A,\Delta\seq C$.  Cut is admissible in $\Gip$
\cite[4.1.9]{TroelstraSchwichtenberg2000}, \cite[Ch.~5]{NegriVonPlato2001}, so
$\Gip\vdash\Gamma,\Delta\seq C$, and Theorem~\ref{thm:g0equiv} again yields
$\Gzi\vdash\Gamma,\Delta\seq C$.
\end{proof}

The route is the one taken by Kamide and Negri \cite[Thm.~2.13]{KamideNegri2025}
and, for the classical block calculus, in \cite{CuconatoKJM}: rather than
eliminate cut in the G0-style calculus, where the independent contexts and the
explicit structural rules multiply the cases, one transfers the derivation to
the structural-rule-free calculus, eliminates there, and transfers back.  The
equivalence theorem is not an end in itself but the instrument of the
elimination.

\begin{rem}[Comparison with Ili\'c's $GI$]\label{rem:ilic}
The calculus $GI$ of \cite{Ilic2016} is the closest published relative of
$\Gzi$.  The two agree on the treatment of implication --- $GI$'s rule
$(\arr\mathrm l)$ is our $\arr\mathrm L^{s}$, with independent contexts and no
repetition of the principal formula, and $(\arr\mathrm r)$ is our
$\arr\mathrm R$ --- and both take the structural rules as primitive.  They
differ in two respects: $GI$ keeps $(\land\mathrm r)$ and $(\lor\mathrm l)$
context-sharing, so that its operational rules are of mixed type, and it
decomposes a conjunction on the left by the two \emph{special} rules
$\Gamma,\alpha\vdash\gamma/\Gamma,\alpha\land\beta\vdash\gamma$ and its twin,
where we use the single rule with both conjuncts in the premiss.  The first
difference is what leads Ili\'c to introduce assumption labels on the natural
deduction side; the second is what makes her elimination rule for conjunction
come in two forms where ours comes in one.  We take up the comparison in
Section~\ref{sec:comp}.
\end{rem}

\subsection{The natural deduction system \texorpdfstring{$\Ngi$}{NgTi}}

\begin{dfn}[The system $\Ngi$]\label{def:ng}
A derivation is a finite tree of formulae grown by the rules below; its topmost
formulae are assumptions, each open or discharged by an application of a rule
further down, and $[A]$ marks an assumption available for discharge.  For a
derivation $\mathcal D$ we write $\oa(\mathcal D)$ for the multiset of its open
assumptions and $\ef(\mathcal D)$ for its end-formula.  The rules are the
following.

\smallskip
\noindent\emph{Introduction rules.}
\[
\frac{A\quad B}{A\land B}\ \land\mathrm I
\qquad
\frac{A}{A\lor B}\ \lor\mathrm I_{1}
\qquad
\frac{B}{A\lor B}\ \lor\mathrm I_{2}
\qquad
\frac{\begin{array}{c}[A]\\ \vdots\\ B\end{array}}{A\arr B}\ \arr\mathrm I
\]

\smallskip
\noindent\emph{General elimination rules.}
\[
\frac{A\land B\quad\begin{array}{c}[A,B]\\ \vdots\\ C\end{array}}{C}\ \land\mathrm E
\qquad
\frac{A\lor B\quad\begin{array}{c}[A]\\ \vdots\\ C\end{array}
\quad\begin{array}{c}[B]\\ \vdots\\ C\end{array}}{C}\ \lor\mathrm E
\]
\[
\frac{A\arr B\quad\begin{array}{c}\vdots\\ A\end{array}
\quad\begin{array}{c}[B]\\ \vdots\\ C\end{array}}{C}\ \arr\mathrm E
\qquad\qquad
\frac{\bot}{C}\ \bot\mathrm E
\]
In each elimination the leftmost premiss is the \emph{major} premiss and the
derivations from bracketed assumptions are the \emph{minor} derivations; in
$\arr\mathrm E$ the derivation of $A$ is a minor derivation without discharge.
\end{dfn}

\begin{dfn}[Full normal form]\label{def:fnf}
A derivation in $\Ngi$ is in \emph{full normal form}, or \emph{fully normal},
if the major premiss of every application of an elimination rule --- including
the premiss $\bot$ of every application of $\bot\mathrm E$ --- is an
assumption, open or discharged.
\end{dfn}

The notion is von Plato's \cite{vonPlato2001,NegriVonPlato2001}, and it is the
general-elimination counterpart of Prawitz's demand that no major premiss of an
elimination be the conclusion of an introduction \cite{Prawitz1965}.  The point
of the general format is that each elimination matches, one for one, a left
rule of $\Gzi$: a left rule has its principal formula in the antecedent, hence,
on the natural deduction side, as an assumption, and this is exactly the major
premiss of the corresponding elimination.  The special eliminations are
recovered as the instances in which the conclusion $C$ is a component and the
minor derivation is the assumption itself.

\begin{dfn}[Subsumption]\label{def:reduct}
A multiset $\Theta$ is \emph{subsumed} by a multiset $\Gamma$, written
$\Theta\sqsubseteq\Gamma$, if every formula occurring in $\Theta$ occurs in
$\Gamma$; equivalently, if the underlying set of $\Theta$ is included in that of
$\Gamma$.
\end{dfn}

Subsumption, and not multiset inclusion, is the right relation between the
antecedent of a sequent and the open assumptions of the corresponding natural
deduction derivation, and the reason is the discipline of discharge.  A rule
that discharges an assumption discharges \emph{all} the occurrences of it that
the derivation carries, so multiplicities on the natural deduction side are not
under the control of the translation: multiple discharge lowers them and
vacuous discharge admits their absence, while the contraction rule of the
sequent calculus identifies two occurrences without any counterpart in the
tree.  Subsumption is closed under all four operations, and is thus preserved by
every case of Theorem~\ref{thm:sctond}.  It is the notion that the
\emph{multiset reduct} of \cite[Def.~3.7]{KamideNegri2025} plays the part of in
the corresponding classical development \cite{CuconatoKJM}, and it is weaker,
as it must be here because contraction is a primitive rule of $\Gzi$.

\subsection{The two translations}

\begin{thm}[Simulation of $\Ngi$ in $\Gzi$]\label{thm:ndtosc}
Let $\mathcal D$ be a derivation in $\Ngi$ with $\oa(\mathcal D)=\Gamma$ and
$\ef(\mathcal D)=C$.  Then $\Gzi\vdash\Gamma\seq C$.  If $\mathcal D$ is fully
normal, the simulating derivation uses no cut.
\end{thm}

\begin{proof}
By induction on the height of $\mathcal D$, with cases on its last rule.  An
assumption $A$, with $\oa=\{A\}$ and $\ef=A$, translates to
$(\mathrm{id}^{s})$.

\emph{Introduction rules.}  If the last rule is $\land\mathrm I$ with premisses
derivations of $A$ from $\Gamma$ and of $B$ from $\Delta$, the induction
hypothesis gives $\Gamma\seq A$ and $\Delta\seq B$, and $\land\mathrm R^{s}$
gives $\Gamma,\Delta\seq A\land B$.  The rules $\lor\mathrm I_{i}$ translate to
$\lor\mathrm R_{i}$.  If the last rule is $\arr\mathrm I$, discharging $k\ge0$
occurrences of $A$ in a derivation of $B$ whose open assumptions are
$A^{k},\Gamma$, the induction hypothesis gives $A^{k},\Gamma\seq B$; if $k\ge1$
we contract to $A,\Gamma\seq B$ by $k-1$ applications of $\ctrL$, and if $k=0$
we weaken to $A,\Gamma\seq B$ by $\wkL$; then $\arr\mathrm R$ gives
$\Gamma\seq A\arr B$.

\emph{General elimination rules.}  These translate to the corresponding left
rules, with a cut on the major premiss.  For $\land\mathrm E$, with major
derivation of $A\land B$ from $\Gamma$ and minor derivation of $C$ from
$[A,B]$ and $\Delta$, the induction hypothesis gives $\Gamma\seq A\land B$ and
$A^{j},B^{k},\Delta\seq C$; contracting and, where a component is discharged
vacuously, weakening, we obtain $A,B,\Delta\seq C$, whence $\land\mathrm L$
gives $A\land B,\Delta\seq C$, and a cut on $A\land B$ gives
$\Gamma,\Delta\seq C$.  For $\lor\mathrm E$ the induction hypothesis gives
$\Gamma\seq A\lor B$, $A^{j},\Delta\seq C$ and $B^{k},\Theta\seq C$, which the
same contractions and weakenings bring to $A,\Delta\seq C$ and
$B,\Theta\seq C$; then $\lor\mathrm L^{s}$ followed by a cut on $A\lor B$ gives
$\Gamma,\Delta,\Theta\seq C$.  For $\arr\mathrm E$, with major derivation of
$A\arr B$ from $\Gamma$, minor derivation of $A$ from $\Delta$ and minor
derivation of $C$ from $[B]$ and $\Theta$, the induction hypothesis gives
$\Gamma\seq A\arr B$, $\Delta\seq A$ and $B^{k},\Theta\seq C$, and, after the
same adjustment on the last, $\arr\mathrm L^{s}$ followed by a cut on
$A\arr B$ gives $\Gamma,\Delta,\Theta\seq C$.  For $\bot\mathrm E$, with premiss a derivation
of $\bot$ from $\Gamma$, the induction hypothesis gives $\Gamma\seq\bot$, and a
cut on $\bot$ against $(\mathrm{L}\bot^{s})$ gives $\Gamma\seq C$.

\emph{Fully normal derivations.}  If $\mathcal D$ is fully normal, the major
premiss of each elimination is an assumption, so the derivation of the major
premiss supplied by the induction hypothesis is the initial sequent
$D\seq D$ with $D$ the principal formula, or, in the case of $\bot\mathrm E$,
the sequent $\bot\seq\bot$.  A cut against $D\seq D$ is redundant: its
conclusion is the other premiss itself, so the left rule may be applied
directly to the antecedent occurrence of $D$, and in the case of
$\bot\mathrm E$ the sequent $\bot\seq C$ is $(\mathrm{L}\bot^{s})$.  Hence no
cut occurs.  In every case $\Gzi\vdash\Gamma\seq C$, cut being admissible by
Theorem~\ref{thm:g0cut}.
\end{proof}

\begin{thm}[Simulation of $\Gzi$ in $\Ngi$]\label{thm:sctond}
If $\Gzi\vdash\Gamma\seq C$, then there is a fully normal derivation
$\mathcal D$ in $\Ngi$ with $\oa(\mathcal D)\sqsubseteq\Gamma$ and
$\ef(\mathcal D)=C$.
\end{thm}

\begin{proof}
By Theorem~\ref{thm:g0cut} we may take the given derivation to be cut-free.  We
argue by induction on its height, with cases on the last rule.

\emph{Initial sequents.}  $A\seq A$ translates to the assumption $A$, a fully
normal derivation with $\oa=\{A\}$ and $\ef=A$.  The sequent $\bot\seq C$
translates to the application of $\bot\mathrm E$ to the assumption $\bot$;
since its major premiss is an assumption, the derivation is fully normal.

\emph{Structural rules.}  Both are translated by leaving the derivation
unchanged.  For $\wkL$, the induction hypothesis gives $\mathcal D$ with
$\oa(\mathcal D)\sqsubseteq\Gamma$, and $\Gamma$ and $A,\Gamma$ have the same
formulae but for $A$, so $\oa(\mathcal D)\sqsubseteq A,\Gamma$; the added
formula simply does not occur among the open assumptions.  For $\ctrL$, the
multisets $A,A,\Gamma$ and $A,\Gamma$ have the same underlying set, so
$\oa(\mathcal D)\sqsubseteq A,A,\Gamma$ gives at once
$\oa(\mathcal D)\sqsubseteq A,\Gamma$.  It is here that subsumption is used in
place of multiset inclusion, and here alone.

\emph{Left rules.}  If the last rule is $\land\mathrm L$, from $A,B,\Delta\seq C$
to $A\land B,\Delta\seq C$, the induction hypothesis gives a fully normal
derivation of $C$ whose open assumptions are subsumed by $A,B,\Delta$; one
application of $\land\mathrm E$, with major premiss the assumption $A\land B$
and that derivation as minor derivation, discharging every open occurrence of
$A$ and of $B$, gives a fully normal derivation of $C$ whose open assumptions
are subsumed by $A\land B,\Delta$, the new major premiss being again an
assumption.  If the last rule is $\lor\mathrm L^{s}$,
the induction hypothesis gives fully normal derivations of $C$ whose open
assumptions are subsumed by $A,\Gamma$ and by $B,\Delta$ respectively, and
$\lor\mathrm E$ with major premiss the assumption $A\lor B$, discharging every
open occurrence of $A$ in the one and of $B$ in the other, concludes with open
assumptions subsumed by $A\lor B,\Gamma,\Delta$.  If the last rule is $\arr\mathrm L^{s}$, from
$\Gamma\seq A$ and $B,\Delta\seq C$, the induction hypothesis gives a fully
normal derivation of $A$ with open assumptions subsumed by $\Gamma$ and a fully
normal derivation of $C$ with open assumptions subsumed by $B,\Delta$; one
application of $\arr\mathrm E$, with major premiss the assumption $A\arr B$, the
first as the minor derivation of $A$ and the second as the minor derivation
discharging every open occurrence of $B$, gives a fully normal derivation of $C$
with open assumptions subsumed by $A\arr B,\Gamma,\Delta$.

\emph{Right rules.}  Each translates to the corresponding introduction:
$\land\mathrm R^{s}$ to $\land\mathrm I$, $\lor\mathrm R_{i}$ to
$\lor\mathrm I_{i}$, and $\arr\mathrm R$, from $A,\Gamma\seq B$ to
$\Gamma\seq A\arr B$, to $\arr\mathrm I$ discharging all open occurrences of
$A$, so that the open assumptions of the result are subsumed by $\Gamma$.  An
introduction adds no elimination, so full normality is preserved.

Since every rule is thus simulated by a step preserving full normality, and the
initial sequents translate to fully normal derivations, the derivation
constructed is fully normal.
\end{proof}

\begin{thm}[Full normalization for $\Ngi$]\label{thm:norm}
For every derivation $\mathcal D$ in $\Ngi$ with $\oa(\mathcal D)=\Gamma$ and
$\ef(\mathcal D)=C$ there is a fully normal derivation $\mathcal D'$ in $\Ngi$
with $\oa(\mathcal D')\sqsubseteq\Gamma$ and $\ef(\mathcal D')=C$.  In
particular, if $\mathcal D$ has no open assumptions, neither has
$\mathcal D'$.
\end{thm}

\begin{proof}
By Theorem~\ref{thm:ndtosc}, $\Gzi\vdash\Gamma\seq C$; by
Theorem~\ref{thm:g0cut} the derivation may be taken cut-free; by
Theorem~\ref{thm:sctond} there is a fully normal $\mathcal D'$ with the stated
open assumptions and end-formula.
\end{proof}

\begin{cor}[The network]\label{cor:network}
For every formula $A$ the following are equivalent:
\textup{(i)} $A$ is provable in $\Ngi$ from no open assumptions;
\textup{(ii)} $\Gzi\vdash{}\seq A$; \textup{(iii)} $\Gip\vdash{}\seq A$;
\textup{(iv)} $\Gti\vdash{}\seq A$; \textup{(v)} the block $\{\sF A\}$ is
refutable in $\bti$; \textup{(vi)} $A$ is valid in all Kripke models.
\end{cor}

\begin{proof}
(i)$\Rightarrow$(ii) is Theorem~\ref{thm:ndtosc} and (ii)$\Rightarrow$(i) is
Theorem~\ref{thm:sctond}, the only multiset subsumed by the empty one being the
empty one.
(ii)$\Leftrightarrow$(iii) is Theorem~\ref{thm:g0equiv},
(iii)$\Leftrightarrow$(iv) is Proposition~\ref{prop:collapse},
(iv)$\Leftrightarrow$(v) is Theorem~\ref{thm:corr}, and
(iv)$\Leftrightarrow$(vi) is Theorem~\ref{thm:complete} with
Corollary~\ref{cor:soundg3}.
\end{proof}

The corollary makes explicit what the sequence of translations achieves.
Closure of a block, cut-freeness of a G0-style derivation and full normality of
a natural deduction derivation are three readings of one combinatorial fact,
and the normalization theorem for $\Ngi$ rests, through
Theorem~\ref{thm:g0cut}, on the cut-freeness that the block calculus enjoys by
construction.

\section{A terminating variant}\label{sec:term}

The rule $\sT\!\arr$ retains its principal formula, and this is what makes
contraction absorbable; it is also what makes naive refutation search
non-terminating, since the same implication may be decomposed again and again.
The remedy is Dyckhoff's \cite{Dyckhoff1992}: replace the single rule for an
implication on the left by four rules, keyed to the shape of its antecedent,
each of which reduces the depth of the block.  We transcribe his calculus
$\Gfip$ --- Hudelmaier's \cite{Hudelmaier1993} independently --- into block
form.

Since $\Gfip$ is a single-succedent calculus, we work with blocks containing
exactly one $\sF$-signed formula, which we call \emph{simple}; by
Corollary~\ref{cor:blockprov} nothing is lost, provability of $A$ being
refutability of the simple block $\{\sF A\}$.

\begin{dfn}[The calculus $\btd$]\label{def:btd}
The blocks of $\btd$ are simple.  Closure is as in
Definition~\ref{def:closure}.  The rules are
\[
\frac{\Pi,\ \sT(A\land B)}{\Pi,\ \sT A,\ \sT B}\ \sT\!\land
\qquad
\frac{\Pi,\ \sT(A\lor B)}{\Pi,\ \sT A\ \mid\ \Pi,\ \sT B}\ \sT\!\lor
\]
\[
\frac{\Pi,\ \sF(A\land B)}{\Pi,\ \sF A\ \mid\ \Pi,\ \sF B}\ \sF\!\land
\qquad
\frac{\Pi,\ \sF(A\lor B)}{\Pi,\ \sF A}\ \sF\!\lor_{1}
\qquad
\frac{\Pi,\ \sF(A\lor B)}{\Pi,\ \sF B}\ \sF\!\lor_{2}
\]
\[
\frac{\Pi,\ \sF(A\arr B)}{\Ptt,\ \sT A,\ \sF B}\ \sF\!\arr
\]
together with the four rules for an implication under $\sT$:
\[
\frac{\Pi,\ \sT p,\ \sT(p\arr B)}{\Pi,\ \sT p,\ \sT B}\ \sT\!\arr_{\mathrm{at}}
\qquad
\frac{\Pi,\ \sT((A\land B)\arr C)}{\Pi,\ \sT(A\arr(B\arr C))}\ \sT\!\arr_{\land}
\]
\[
\frac{\Pi,\ \sT((A\lor B)\arr C)}{\Pi,\ \sT(A\arr C),\ \sT(B\arr C)}\ \sT\!\arr_{\lor}
\qquad
\frac{\Pi,\ \sT((A\arr B)\arr C)}
     {\Ptt,\ \sT(B\arr C),\ \sF(A\arr B)\ \ \mid\ \ \Pi,\ \sT C}\ \sT\!\arr_{\arr}
\]
In $\sT\!\arr_{\mathrm{at}}$ the letter $p$ is a propositional variable, which
must already occur under the sign $\sT$ in the block and is retained; in
$\sT\!\arr_{\arr}$ the $\Ptt$ of the left child is computed after the removal
of the principal formula.  No rule applies to $\sT(\bot\arr C)$, nor to
$\sT(p\arr B)$ when $\sT p$ is absent: such formulae are inert.
\end{dfn}

Simplicity is preserved by every rule: the rules with a $\sT$-signed principal
formula leave the $\sF$-part untouched; $\sF\!\land$ and $\sF\!\lor_{i}$ replace
the single $\sF$-signed formula by one of its components in each child;
$\sF\!\arr$ replaces it by $\sF B$; and $\sT\!\arr_{\arr}$ replaces it by
$\sF(A\arr B)$ in the left child and retains it in the right one.  The two rules
$\sF\!\lor_{i}$ are alternatives, not branchings: a tableau applies one or the
other.  The rule $\sT\!\arr_{\arr}$ is the only one of the
four that branches, and its left child performs a purge, since the succedent of
the corresponding premiss of $\Gfip$ is replaced by $A\arr B$.  On simple
blocks the rule $\sF\!\arr$ purges nothing, the only $\sF$-signed formula being
the principal one; the purge survives in $\btd$ solely in
$\sT\!\arr_{\arr}$.

\begin{prop}[$\btd$ is $\Gfip$ read upside down]\label{prop:btdg4}
For every simple block $\Pi=\sT[\Gamma]\cup\{\sF C\}$: $\Pi$ is refutable in
$\btd$ if and only if $\Gamma\seq C$ is derivable in Dyckhoff's calculus
$\Gfip$ \textup{\cite[\S3]{Dyckhoff1992}}.
\end{prop}

\begin{proof}
The rules of Definition~\ref{def:btd} are, one by one, the rules of $\Gfip$
read from conclusion to premisses under the translation
$\Pi=\sT[\Gamma]\cup\{\sF C\}$, exactly as in the proof of
Theorem~\ref{thm:corr}: $\sT\!\land$, $\sT\!\lor$, $\sF\!\land$,
$\sF\!\lor_{i}$ and $\sF\!\arr$ correspond to $\land\mathrm L$,
$\lor\mathrm L$, $\land\mathrm R$, $\lor\mathrm R_{i}$ and $\arr\mathrm R$, and
the four rules $\sT\!\arr_{\ast}$ correspond to the four rules
$\mathrm L\arr_{\mathrm{at}}$, $\mathrm L\arr_{\land}$, $\mathrm L\arr_{\lor}$,
$\mathrm L\arr_{\arr}$ of $\Gfip$; closure on $\sT p,\sF p$ corresponds to the
initial sequents $p,\Gamma\seq p$ and closure on $\sT\bot$ to $\bot,\Gamma\seq
C$.  The verification is a rule-by-rule inspection, and the adjunction of
copies needed to pass from sets to multisets is licensed, as before, by the
admissibility of weakening in $\Gfip$ \cite[Lemma~1]{Dyckhoff1992}.
\end{proof}

\begin{thm}[Termination]\label{thm:term}
Define a weight function on formulae by
\[
\wt(p)=\wt(\bot)=1,\qquad
\wt(A\land B)=\wt(A)+\wt(B)+2,
\]
\[
\wt(A\lor B)=\wt(A)+\wt(B)+1,\qquad
\wt(A\arr B)=\wt(A)+\wt(B)+1,
\]
and let $\mu(\Pi)$ be the multiset of the weights of the formulae occurring in
$\Pi$, a signed formula $\sT A$ or $\sF A$ contributing one occurrence of
$\wt(A)$.  Then every rule of $\btd$ strictly decreases $\mu$ in the multiset
ordering on natural numbers.  Consequently no branch of a $\btd$ tableau is
infinite, every $\btd$ tableau is finite, and exhaustive refutation search in
$\btd$ terminates.
\end{thm}

\begin{proof}
Write $a=\wt(A)$, $b=\wt(B)$, $c=\wt(C)$.  In the multiset ordering, a step
that removes a multiset $X$ and adds a multiset $Y$ decreases the measure
provided every element of $Y$ is strictly smaller than some element of $X$; we
verify this for each rule, noting that the passage from a block to its child
removes at least the principal formula and adds at most the displayed ones.
Since blocks are sets, a formula already present is not added twice, which can
only decrease the measure further.

$\sT\!\land$ removes the weight $a+b+2$ and adds $a$ and $b$, both smaller.
$\sT\!\lor$ and $\sF\!\lor_{i}$ remove $a+b+1$ and add $a$ or $b$.
$\sF\!\land$ removes $a+b+2$ and adds $a$ or $b$.  $\sF\!\arr$ removes
$a+b+1$, together with any further $\sF$-signed formulae purged, and adds $a$
and $b$.  $\sT\!\arr_{\mathrm{at}}$ removes $\wt(p\arr B)=b+2$ and adds $b$,
the formula $\sT p$ being retained unchanged.

$\sT\!\arr_{\land}$ removes
$\wt((A\land B)\arr C)=(a+b+2)+c+1=a+b+c+3$ and adds
$\wt(A\arr(B\arr C))=a+(b+c+1)+1=a+b+c+2$, which is strictly smaller.  It is
for this rule that conjunction must be given weight $\wt(A)+\wt(B)+2$; with
weight $\wt(A)+\wt(B)+1$ the two sides would be equal.

$\sT\!\arr_{\lor}$ removes $\wt((A\lor B)\arr C)=(a+b+1)+c+1=a+b+c+2$ and adds
$\wt(A\arr C)=a+c+1$ and $\wt(B\arr C)=b+c+1$; since $b\ge1$ and $a\ge1$, both
are strictly smaller than $a+b+c+2$.  Note that the \emph{sum} of the two
weights need not be smaller; it is the multiset ordering, and not the sum, that
decreases.

$\sT\!\arr_{\arr}$ removes $\wt((A\arr B)\arr C)=(a+b+1)+c+1=a+b+c+2$.  In the
left child it adds $\wt(B\arr C)=b+c+1$ and $\wt(A\arr B)=a+b+1$, both strictly
smaller than $a+b+c+2$ because $a\ge1$ and $c\ge1$, and it removes in addition
the $\sF$-signed formula of the block.  In the right child it adds $c$, again
strictly smaller.

The multiset ordering on finite multisets of natural numbers is well founded
\cite{DershowitzManna1979}, so no branch admits an infinite descending sequence
of measures, hence no branch is infinite.  A $\btd$ tableau is finitely
branching, every rule having at most two children, so by K\"onig's lemma it is
finite.
\end{proof}

\begin{thm}[Adequacy of $\btd$ and decidability]\label{thm:btdadeq}
A simple block $\Pi=\sT[\Gamma]\cup\{\sF C\}$ is refutable in $\btd$ if and
only if it is refutable in $\bti$.  Hence refutation search in $\btd$ is a
terminating decision procedure for intuitionistic propositional logic.
\end{thm}

\begin{proof}
By Proposition~\ref{prop:btdg4}, $\Pi$ is refutable in $\btd$ if and only if
$\Gfip\vdash\Gamma\seq C$.  By Dyckhoff's theorem \cite[Thm.~2]{Dyckhoff1992},
$\Gfip$ and $\LJ$ --- equivalently $\Gip$ --- derive the same sequents; see also
\cite{Hudelmaier1993} and \cite{DyckhoffNegri2000}.  By
Corollary~\ref{cor:blockprov} and Theorem~\ref{thm:corr}, $\Gip\vdash\Gamma\seq
C$ if and only if $\Pi$ is refutable in $\bti$.  Termination is
Theorem~\ref{thm:term}, and since a tableau is finite and each rule application
is effective, exhaustive search decides refutability.
\end{proof}

We have deliberately obtained the adequacy of $\btd$ from a published theorem
rather than reproved it: the syntactic simulation of $\arr\mathrm L$ in
$\Gfip$, which is what the difficult direction amounts to, is a substantial
argument, and its transcription into block notation would add length without
adding insight.  What the block presentation does contribute is the
verification of Theorem~\ref{thm:term}, in which the role of each rule in the
descent of the measure is visible on the face of the block.

\section{Blocks, sequents and natural deduction compared}\label{sec:comp}

The systems of the preceding sections prove the same theorems, and
Corollary~\ref{cor:network} says so.  What they do not share is the way in
which they say it, and it is worth setting the differences out, because the
intuitionistic case brings to light distinctions that the classical case
conceals.  We proceed along eight axes.

\subsection{The datum and its orientation}

A block is a finite set of signed formulae and is read by \emph{analysis}: one
begins with the signed formulae expressing the refutation to be carried out and
descends, decomposing.  A sequent is a pair of multisets and is read by
\emph{synthesis}: one begins with initial sequents and ascends, composing.  A
natural deduction derivation occupies the middle ground, in which the
eliminations decompose while the introductions compose, and full normal form is
precisely the regime in which the two do not interfere.

The three notations distribute the same information differently.  What the
sequent expresses by the position of a formula with respect to the arrow, the
block expresses by a sign attached to the formula itself, and natural deduction
by the distinction between an assumption and a conclusion.  This is not a
matter of taste in the intuitionistic setting.  In the classical block calculus
the sign is dispensable, because the involutive negation can carry it: a
formula is asserted if it appears unnegated and denied if it appears negated,
and $\lnot\lnot A$ collapses to $A$ so that the two states exhaust the
possibilities.  Intuitionistically the collapse fails, $\lnot\lnot A$ being
weaker than $A$, and the sign must be made explicit.  Signs are thus not a
notational convenience imported from Smullyan but a necessity imposed by the
logic.

\subsection{Where intuitionism resides}

In $\bti$ the whole of the intuitionistic restriction is the purge, and the
purge occurs in one rule, $\sF\!\arr$ (with its derived instance $\sF\!\lnot$).
In $\Gti$ it is the discarding of the succedent context in $\arr\mathrm R$.  In
$\Gip$, $\Gzi$ and $\LJ$ it is not a feature of any rule but of the
\emph{format}: there is only ever one formula in the succedent, so nothing is
left to discard.  In $\Ngi$ and in $\NJ$ it is again a feature of the format,
namely that a derivation has one conclusion, together with the absence of the
classical reductio.

These four descriptions are equivalent, by the translations of
Sections~\ref{sec:g3}--\ref{sec:nd}; but they are not equally informative.  The
single-succedent presentations record the restriction by refusing to state the
alternative, so that one cannot see, inside the calculus, what would have to be
given up for the logic to become classical.  The multiple-succedent
presentations state the alternative and then forbid its transport, and this is
why the difference between the classical and the intuitionistic calculus can be
localized in a single deletion, as Example~\ref{ex:peirce} shows.  The block
calculus inherits the second, more informative option, since a block carries
several $\sF$-signed formulae as a matter of course.

\subsection{The structural rules}

The four systems exhibit the four possible attitudes towards weakening and
contraction.  In $\bti$ they are \emph{absorbed}: contraction into the
set-theoretic reading of blocks and into the persistence of $\sT\!\arr$,
weakening into the closure criterion.  In $\Gti$ and $\Gip$ they are
\emph{height-preserving admissible}, which is the sequent-calculus form of the
same fact.  In $\Gzi$, and in Ili\'c's $GI$, they are \emph{primitive}.  In
$\Ngi$ they are \emph{implicit} in the discipline of discharge: contraction is
multiple discharge, weakening is vacuous discharge.

It is at this point that the comparison with \cite{Ilic2016} is most
instructive.  Ili\'c replaces the initial formulae of natural deduction by
\emph{initial rules} $\alpha,\alpha_{1},\dots,\alpha_{n}/\alpha$, with the
effect that no assumption is ever discharged vacuously, every formula
discharged by a rule really occurring in the upper part of some initial rule;
and she uses assumption labels to control the identification of occurrences.
Her derivations are then normal by construction, and by restricting the initial
rules or forbidding multiple discharge one obtains natural deduction systems
for $BCK$ and for a relevant logic.  The device and the absorption of the block
calculus pursue the same end from opposite directions.  Ili\'c makes the
structural rules \emph{visible and controllable}, so that they can be dropped
one at a time; the block calculus makes them \emph{invisible}, so that they
need never be mentioned.  The two are complementary rather than rival, and each
pays for its advantage: a calculus in which contraction has been absorbed into
the data structure cannot be weakened to a contraction-free logic by deleting a
rule, since there is no rule to delete, whereas a calculus that carries the
structural rules explicitly must prove their admissibility elsewhere or forgo
analyticity.

\subsection{Contraction on the implication, in four guises}\label{subsec:four}

One phenomenon recurs in four notations, and recognizing it as one phenomenon
is, we think, the main expository gain of the comparison.  In $\bti$ the rule
$\sT\!\arr$ retains its principal formula in the child that carries $\sF A$.
In $\Gti$ and $\Gip$ the rule $\arr\mathrm L$ repeats the principal formula in
the left premiss.  In $\Gzi$ the repetition is gone and its work is done by the
primitive $\ctrL$.  In $\Ngi$ it is the multiple discharge of the assumption
$A\arr B$ by the several applications of $\arr\mathrm E$ that use it.

The four are related by the translations already given, and each is the
efficient form for its own purpose: persistence for a calculus without
structural rules, repetition for the admissibility proof of contraction,
explicit contraction for the correspondence with discharge, multiple discharge
for the natural deduction reading.  It should also be said that this is the
feature that costs the most.  It is the persistence of $\sT\!\arr$ that makes
naive refutation search diverge, and Section~\ref{sec:term} is entirely devoted
to buying it out.

\subsection{Analyticity and normal forms}

Analyticity is realized in three modes.  In $\bti$ it holds \emph{by
construction}: every signed formula produced by a rule is a signed immediate
subformula of the principal one, so Proposition~\ref{prop:subf} is a matter of
inspection.  In $\Gti$ and $\Gip$ it holds \emph{by theorem}: the calculi admit
cut and the subformula property follows once cut has been shown eliminable.  In
$\Ngi$ it holds \emph{by normalization}: Theorem~\ref{thm:norm} brings every
derivation to a form in which no elimination is applied to a formula that has
just been introduced.  Corollary~\ref{cor:network} says that the three are
readings of one fact.

The intuitionistic case adds an asymmetry with no classical counterpart.  In
the classical calculi the multiple-succedent format is the natural one and cut
elimination goes through in it by the usual permutations.  Here, as
Remark~\ref{rem:cutobstruction} shows, the permutation of cut past
$\arr\mathrm R$ in the multiple-succedent calculus fails, precisely because
that rule empties the succedent; the obstruction vanishes in the
single-succedent calculus, where there is nothing to empty.  The rule that
carries the intuitionistic content is thus also the rule that obstructs the
standard syntactic treatment of cut, and one is driven either to a semantic
argument, as in Corollary~\ref{cor:semcut}, or to the single-succedent detour,
as in Theorem~\ref{thm:g0cut}.

\subsection{Invertibility and proof search}

By Lemma~\ref{lem:g3inv} and Theorem~\ref{thm:corr}, every rule of $\bti$ is
invertible except $\sF\!\arr$; and $\sF\!\arr$ is the intuitionistic rule.  The
same holds on the sequent side for $\arr\mathrm R$ and, on the natural deduction
side, for $\arr\mathrm I$.  The proposition is worth stating in the form it
takes for search: in a refutation, all choices are \emph{don't-care} except the
choice of which $\sF$-signed implication to decompose, which is
\emph{don't-know}, and except the branchings, which must all close.  The whole
combinatorial difficulty of intuitionistic proof search is concentrated at the
purge, and the reason is now visible: an application of $\sF\!\arr$ destroys
information --- the other $\sF$-signed formulae --- and destroyed information
cannot be recovered by backtracking within the branch.

In $\btd$ the situation is slightly different, since $\sT\!\arr_{\arr}$ is not
invertible either: its left child replaces the $\sF$-signed formula.  In
sequent notation, $(p\arr q)\arr r,\,r\seq r$ is an initial sequent, while the
left premiss that $\sT\!\arr_{\arr}$ would require, $q\arr r,\,r\seq p\arr q$,
is falsified by the model $w\le v$ with $V(w)=\{r\}$ and $V(v)=\{r,p\}$.  The calculus buys termination at the
cost of a second point of genuine choice.  This is the trade-off that Dyckhoff's calculus makes, and
it is the same trade-off in block notation.

\subsection{Termination, decidability, and the size of refutations}

Refutation search in $\bti$ need not terminate, and
Corollary~\ref{cor:dec} obtains decidability only through the finite model
property, that is, by a detour through semantics.  In $\Gti$ the same
difficulty is met by loop-checking, in $\Ngi$ by a measure on normalization,
and in $\btd$ by the depth-reducing rules, for which Theorem~\ref{thm:term}
supplies an explicit well-founded measure.  Of the three remedies only the last
is a genuine syntactic decision procedure, and only for the fragment of simple
blocks --- which, by Corollary~\ref{cor:blockprov}, is no restriction.

Against these gains one must set what analyticity costs, and the accounting is
the same here as in the classical case, only worse.  Cut-free and normal proofs
are not always faithful to mathematical reasoning as practised, which composes
lemmas rather than decomposing a thesis; the elimination of cut can increase the
size of a derivation enormously \cite{Boolos1984}; and analytic propositional
calculi admit, in general, no short proofs \cite{Statman1978}.  It was in
response to these anomalies that the line of \emph{analytic cut} was developed
\cite{DAgostinoMondadori1994}, in which cut is restricted to subformulae of the
conclusion rather than eliminated.  To this the intuitionistic case adds a
difference of complexity class that no choice of calculus can remove:
derivability in intuitionistic propositional logic is PSPACE-complete
\cite{Statman1979}, where classical propositional validity is only coNP-complete.
The block calculus inherits both sides of the balance --- full controllability
of derivations, and no promise of short ones.

\subsection{Countermodels and meaning}

A completed open branch of a $\bti$ tableau yields a finite Kripke countermodel
directly, by Corollary~\ref{cor:completebt}, and the worlds of the countermodel
are the segments of the branch between two applications of $\sF\!\arr$: the
purge is not merely a syntactic restriction but the syntactic trace of the
passage to a later world.  Examples~\ref{ex:lem} and \ref{ex:peirce} exhibit
the reading in the two cases that matter.  In the sequent calculi the
countermodel is extracted from a failed proof search, which is the same
information presented upside down; in natural deduction there is no natural
extraction at all.

There is a final point, of a different order, on which the three formalisms
differ, and it is one on which the block calculus has something to say to
current work in proof-theoretic semantics.  $\bti$ has no introduction rules.
It is a purely eliminative presentation: every rule takes a signed compound
apart, and the meaning of a connective, so far as the calculus displays it,
consists in what may be done with a formula of that form.  This is the
consequentialist reading of validity that Schroeder-Heister and, more recently,
Gheorghiu and Pym \cite{GheorghiuPym2026} have developed in the setting of
base-extension semantics, following Sandqvist \cite{Sandqvist2015}.  Harmony,
by contrast, is not visible in $\bti$ at all, there being nothing for the
eliminations to be in harmony with; it becomes visible only once the
introduction rules are restored, in $\Gzi$ and $\Ngi$.  The network of
Section~\ref{sec:nd} is, from this angle, what supplies the block calculus with
a theory of meaning: it erects the constructive scaffolding of introductions and
normal forms above an analytic base that has none.

\subsection{Transport, reiteration, and the size of derivations}

There is a format of natural deduction that we have not so far brought into the
comparison, and it turns out to be the one that the block calculus most closely
resembles on the very point at which it differs from Gentzen's.  In the
Ja\'skowski--Fitch presentation \cite{Jaskowski1934,Fitch1952} a derivation is a
sequence of lines with nested subproofs, and a formula proved outside a subproof
becomes available inside it only by an explicit act of \emph{reiteration}. The system is then tuned by restricting what may be reiterated.  As Hazen and
Pelletier observe \cite{HazenPelletier2014}, the modal logics $\mathsf T$,
$\mathsf{S4}$, $\mathsf{S5}$ and $\mathsf B$ are obtained from one another by
varying what may be reiterated into a strict subproof, and in what form.  The
corresponding Gentzen-style systems must impose the same restrictions as side
conditions on the undischarged hypotheses standing above an inference --- a
formulation they judge, with reason, less perspicuous.

The transition rule $\sF\!\arr$ is a restriction on reiteration in precisely
this sense.  Passing to the child block opens what Fitch would call a strict
subproof, and the restriction reads: only $\sT$-signed formulae may be
reiterated into it.  The point is not a matter of analogy.  It supplies the
uniform template along which the calculus can be varied, and the variations of
Remark~\ref{rem:neigh} are instances of it: relaxing the restriction so as to
admit some $\sF$-signed formulae, under conditions on their shape, moves one
into the intermediate logics; adding a second kind of transition, with its own
restriction, is what a modal extension will require.  A condition on what
crosses a single node is easier to state, and to vary, than a condition on the
undischarged hypotheses of a whole subtree.

The comparison also has a quantitative side, and it is best stated precisely
rather than gestured at. The two translations of Section~\ref{sec:nd} are each
linear in the size of their input: both are step-by-step simulations, and
neither duplicates a subderivation.  The whole of the growth in Theorem~\ref{thm:norm} is therefore confined to the cut-elimination step, and whatever bound holds there is inherited unchanged, the translations contributing none of their own.  Full normalization should be read
as a transformation theorem and not as a statement about the size of normal
derivations.  As for the derivations themselves, a block and a sequent derivation carry the same information at the same cost, each node rewriting its entire context, whereas a linear derivation records the context once and refers back to it by line number.  The block
format buys the absence of every kind of bookkeeping --- no discharge labels, no
line numbers, no reiteration steps --- at the price of transporting the context
in full, and it is a fair exchange only because the context is a set, so that
what is transported never grows by repetition.

\subsection{Summary}

The distinctions drawn in this section are collected in
Table~\ref{tab:compare}, which should be read with two cautions.  A table
registers differences but not their weight, and here the weight is unevenly
distributed: the first three rows fix the format, and most of what follows is a
consequence of them rather than an independent datum.  And the fourth row is not
to be read as the others are.  It does not record four features but one
combinatorial phenomenon under four names, as Section~\ref{subsec:four} argued;
the entries are notational variants of each other, whereas the entries of the
fifth row are genuinely different places at which one and the same restriction
can be lodged.

\begin{table}[htb]
\centering
\footnotesize
\setlength{\tabcolsep}{4pt}
\begin{tabular}{@{}lllll@{}}
\toprule
 & $\bti$ & $\Gti$ & $\Gzi$ & $\Ngi$ \\
\midrule
datum & signed block & multi-succ.\ sequent & single-succ.\ sequent & derivation tree \\
orientation & analysis & synthesis & synthesis & mixed \\
structural rules & absorbed & hp-admissible & primitive & in the discharge \\
contraction on $\arr$ & persistence & repetition & $\ctrL$ & multiple discharge \\
intuitionism resides in & the purge & $\arr\mathrm R$ empties $\Delta$ & the format & the format \\
non-invertible rule & $\sF\!\arr$ & $\arr\mathrm R$ & $\arr\mathrm R$ & $\arr\mathrm I$ \\
analyticity & by construction & by cut elim.\ & by cut elim.\ & by normalization \\
countermodels & from open branches & from failed search & from failed search & --- \\
\bottomrule
\end{tabular}
\caption{The four presentations compared.  The upper rows fix the format, the
lower ones record the proof-theoretic behaviour that follows from it.  Read from
left to right, the fifth row shows the locus of intuitionism migrating from a
rule of the calculus to a restriction on the shape of its sequents and
derivations, and the sixth shows that wherever it is lodged it is the same rule,
under its several names, that fails to be invertible.}
\label{tab:compare}
\end{table}

\section{Concluding remarks}\label{sec:remarks}

\begin{rem}[The first-order extension]\label{rem:fo}
The first-order block calculus of \cite{CuconatoIM} adds the rules of type
$\gamma$, which instantiate a universal formula on a parameter of the branch
and retain the principal formula, and of type $\delta$, which instantiate on a
fresh parameter and consume it.  In the intuitionistic setting these become the
rules for $\sT\forall$, $\sF\exists$ (of type $\gamma$) and $\sT\exists$,
$\sF\forall$ (of type $\delta$), with the proviso that $\sF\forall$ is also a
rule of transition: falsifying a universal formula may require passing to a
later world, and its rule must therefore purge the $\sF$-part exactly as
$\sF\!\arr$ does.  Two restrictions of different origin then coexist in the same
rule, freshness of the parameter and deletion of the context, and it would be
worth asking whether they can be given a common measure.  The correspondence
with the first-order multiple-succedent calculus and with general-elimination
natural deduction extends along the lines of Section~\ref{sec:nd}; termination,
of course, does not.
\end{rem}

\begin{rem}[The modal reading]\label{rem:modal}
The observation of Section~\ref{sec:intro}, that the intuitionistic
propositional rules behave like quantificational or modal rules, is more than
an analogy.  Under the G\"odel--McKinsey--Tarski translation \cite{Godel1933}
an intuitionistic formula is mapped into $\mathsf{S4}$ by prefixing $\Box$ to
its subformulae, and the purge of $\sF\!\arr$ becomes the restriction on the
premiss of $\Box\mathrm R$, the rule that in a cut-free $\mathsf{S4}$ calculus
retains only the boxed formulae of the antecedent.  The same restriction is
what the labelled treatment of modal logic replaces by an explicit relational
atom \cite{Negri2005}, and it is what the labelled block calculus records by a
change of label.  A labelled variant of $\bti$, in which the purge is replaced
by the introduction of a fresh label $v$ with $w\le v$, is therefore available;
it trades the absorption of the structural rules for direct access to the
semantics, and we leave its development for another occasion.
\end{rem}

\begin{rem}[Neighbouring logics]\label{rem:neigh}
Two directions of variation are open, and by Section~\ref{sec:comp} both are
variations on a single parameter, the restriction governing what may cross a
transition.  Weakening the restriction --- allowing some $\sF$-signed formulae
to be retained, under conditions on their shape --- yields the intermediate
logics; the calculus for the logic of
G\"odel and Dummett and that for Jankov's logic are the natural first cases, and
they can be read off the corresponding multiple-succedent calculi through
Theorem~\ref{thm:corr}.  Restricting the discharge on the natural deduction
side, in the manner of Ili\'c's initial rules \cite{Ilic2016}, yields the
substructural neighbours; the block calculus cannot follow it there, for the
reason given in Section~\ref{sec:comp}, and this marks the boundary of the
present approach rather than a defect of it.
\end{rem}

Let us close on the methodological point.  The classical block calculus was
designed for a single purpose, to combine the readability of sequent notation
with the branching economy of tableaux, and there was no reason to expect it to
survive the passage to a logic in which the two features it relies on ---
involutive negation and unrestricted transport of the context --- both fail.
It survives because those two features turn out to be separable from the design
and replaceable, the first by the sign discipline and the second by the purge.
That the resulting calculus differs from its classical ancestor in exactly one
respect, and that this respect is exactly what Kripke semantics prescribes, is
the sort of agreement between a formalism and its intended reading that makes
the formalism an instrument of analysis and not merely of notation.

\end{document}